\documentclass[11pt]{article}

\usepackage[letterpaper,margin=1in]{geometry}
\usepackage{amsmath}
\usepackage{amsthm}
\usepackage[colorlinks=true,allcolors=blue]{hyperref}
\usepackage{cleveref}

\theoremstyle{plain}
\newtheorem{theorem}{Theorem}[section]
\newtheorem{lemma}[theorem]{Lemma}
\newtheorem{corollary}[theorem]{Corollary}
\newtheorem{proposition}[theorem]{Proposition}

\numberwithin{equation}{section}

\newcommand{\newsiamremark}[2]{%
  \theoremstyle{remark}\newtheorem{#1}[theorem]{#2}\theoremstyle{plain}}
\newcommand{\headers}[2]{}
\newenvironment{keywords}%
  {\par\medskip\noindent\textbf{Key words. }}{\par\medskip}
\newenvironment{MSCcodes}%
  {\par\noindent\textbf{MSC codes. }}{\par\medskip}

\usepackage[english]{babel}
\usepackage{amsfonts,amssymb}
\usepackage{mathtools}
\usepackage{graphicx}
\usepackage{enumitem}
\usepackage{soul}
\usepackage{color}

\renewcommand{\underline}[1]{\ul{#1}}

\newsiamremark{remark}{Remark}
\newsiamremark{example}{Example}

\DeclareMathOperator{\Lip}{Lip}
\DeclareMathOperator{\argmin}{arg\,min}
\DeclareMathOperator{\argmax}{arg\,max}

\definecolor{darkblue}{rgb}{0.1,0.1,0.6}

\def\bbR{\mathbb{R}}

\def\eps{\epsilon}
\def\dd{\mathrm{d}}
\def\cP{\mathcal{P}}
\def\weakstarto{\overset{\ast}{\rightharpoonup}}
\def\cE{\mathcal{E}}
\def\Glim{\Gamma\mathrm{-}\lim}

\setlist[enumerate]{leftmargin=.5in}
\setlist[itemize]{leftmargin=.5in}

\headers{Sandwich Slice}{R. D\'iaz Mart\'in, X. Liu, M. Thorpe, and S. Kolouri}

\title{Sandwich Slice: Optimal Transport Potentials Are Optimal Generalized Slicers
}

\author{Roc\'io D\'iaz Mart\'in\thanks{Department of Mathematics, Florida State University.}
\and Xinran Liu\thanks{Harvard T.H. Chan School of Public Health}
\and Matthew Thorpe\thanks{Department of Statistics, University of Warwick.}
\and Soheil Kolouri\thanks{College of Connected Computing, Vanderbilt University.}}

\hypersetup{
  pdftitle={Sandwich Slice: Optimal Transport Potentials Are Optimal Generalized Slicers},
  pdfauthor={Rocio Diaz Martin, Xinran Liu, Matthew Thorpe, and Soheil Kolouri}
}

\begin{document}

\maketitle

\begin{abstract}
Sliced optimal transport replaces a transport problem in $\mathbb{R}^d$ by
one-dimensional problems along scalar slices, and has been used both to lower
bound the Wasserstein distance $W_1$ (max-sliced distances) and, more recently,
to upper bound it by lifting a one-dimensional optimal plan back to the ambient
supports (min-sliced plans). We show that for $1$-Lipschitz slicers $f$ these two
directions are two sides of a single sandwich, $L(f) \le W_1 \le 
U(f)$, 
and we characterize when it collapses to equality.
If the slicer class contains an
optimal Kantorovich--Rubinstein (KR) potential $\phi^*$, then $L(\phi^*) = W_1 =
U(\phi^*)$: the \emph{same} slicer simultaneously maximizes the
max-sliced lower objective and minimizes the min-sliced upper objective, and
both equal $W_1$. 
We define the upper objective through a tie-broken transport
problem, 
which removes the dependence on the one-dimensional ordering and
yields the identities with no uniqueness assumption on the slice-optimal
coupling. 
The result requires the slicer to be nonlinear in general,
consistent with the known failure of exact linear max-sliced/$W_1$ equivalence
in dimension $d \ge 2$, and identifies KR potentials as optimal generalized
slicers.
\end{abstract}






\section{Introduction}
\label{sec:intro}

Optimal transport (OT) provides a geometrically faithful way to compare
probability measures, but its computational cost in high dimensions has long
motivated cheaper surrogates. One of the most prominent of these is the
\emph{sliced} family: rather than solving a transport problem in $\mathbb{R}^d$,
one projects both measures onto the real line through a one-dimensional map,
solves the resulting (closed-form) one-dimensional OT problem, and aggregates
over a collection of such maps. With linear projections, this yields the sliced
Wasserstein (SW) distance of Rabin et al.~\cite{rabin2012wasserstein}, also studied by Bonnotte in his doctoral thesis~\cite{Bonnotte}, and by Bonneel et
al.~\cite{bonneel2015sliced}. 
Deshpande et
al.~\cite{deshpande2019maxsliced} introduced the max-sliced Wasserstein
distance, which selects a single most-discriminative projection rather than
averaging, and Paty and Cuturi~\cite{paty2019subspace} studied a max-min,
subspace-robust variant projecting onto an optimal low-dimensional subspace.
Nguyen et al. \cite{nguyen2021distributional} propose the Distributional Sliced-Wasserstein (DSW) distance, which can be viewed as lying between the classical SW and max-SW (rather than sampling directions from the uniform distribution, as in SW, or optimizing over a single maximizing direction, as in max-SW, DSW optimizes over a regularized probability distribution on projection directions).
Kolouri et
al.~\cite{kolouri2019generalized} replaced the linear projection by a nonlinear
\emph{defining function}, producing the generalized sliced Wasserstein (GSW)
distance and its max-sliced variant. Important precursors of the slicing approach are the iterative transfer ideas developed by Pitié et al.  \cite{pitie2005n, pitie2007automated}.
SW constructions have also found applications in topological data analysis \cite{carriere2017sliced},
statistical and parametric estimation \cite{kolouri2018sliced}, generative modeling \cite{deshpande2018generative,kolouri2019sliced,liutkus2019sliced,wu2019sliced}, domain adaptation \cite{lee2019sliced}
and on imaging and geometric data analysis, including point-cloud registration, the study of sliced Wasserstein gradient flows and texture synthesis (see, e.g., \cite{lai2017multiscale,cozzi2025long, heitz2021sliced}). In this broader context, we ask when a single, possibly nonlinear, scalar slicer simultaneously recovers the ambient Wasserstein distance value and admits an ambient-optimal plan among its slice-optimal couplings.

\paragraph{Two roles for slicing, and two directions of bound}
Slicing has been used for two distinct purposes. The first is to define a cheap
\emph{distance}: sliced, max-sliced, and generalized sliced Wasserstein distances furnish \emph{lower} bounds on the
ambient Wasserstein distance, so long as each slicer is a $1$-Lipschitz function.
The second, more recent, is to produce a \emph{transport plan}: a one-dimensional
optimal coupling can be lifted back to the ambient supports, and the ambient cost
of that lifted plan is an \emph{upper} bound on the Wasserstein distance.
Mahey et al.~\cite{mahey2023swgg} made this upper-bound role explicit with Sliced Wasserstein Generalised Geodesics 
(SWGG) discrepancy and its min-SWGG variant, which lifts the plan induced by an optimal one-dimensional linear projection
and selects the slice of least ambient cost; they observe directly that
min-SWGG is an upper bound, while max-SW is a lower bound. Chapel et
al.~\cite{chapel2025dgswp} extended this to nonlinear slicers (min-GSWP),
formulated slice selection as a bilevel optimization, and again noted that the
lifted coupling is an admissible ambient plan, hence an upper bound on $W_p$ (for $p>1$). Inspired by the work of Rowland et al. \cite{rowland2019orthogonal}, 
Liu et al.~\cite{liu2025expected} construct ambient couplings by averaging,
rather than minimizing over, lifted one-dimensional plans (expected sliced
transport plans), and Tanguy et al.~\cite{tanguy2025sliced} analyze both the
pivot/min-sliced (a generalization of SWGG) and expected-sliced constructions for general measures and
relate them to constrained Kantorovich formulations. Thus, both the lower
(max-sliced) and upper (min-sliced) directions, including for nonlinear
slicers, are by now established.

Moreover, prior work already establishes exactness of an optimized min-sliced upper construction in a particular high-dimensional $W_2$ regime. Specifically, Mahey et al. \cite[Proposition~3.2]{mahey2023swgg}, as restated by Tanguy et al. \cite[Section~6.2, Proposition~11]{tanguy2025sliced}, show that for two uniform $n$-point measures (with an additional technical hypothesis), 
the condition $2n\leq d+1$ ensures that the min-SWGG coincides with the ambient $W_2$ distance.

Our question is different. Rather than identifying classes of measure pairs for which optimization over projections recovers the ambient Wasserstein distance, our goal is to characterize, for a fixed pair of measures, the possibly nonlinear (1-Lipschitz) slicers that make both roles of slicing exact. Namely, we focus our analysis around the $W_1$ distance, for which we address that the induced one-dimensional $W_1$
 distance must recover the ambient $W_1$
 distance, while the slicer-induced transport problem must admit an ambient-optimal coupling (i.e., among the couplings that are optimal for the sliced-based ground cost, there is one that is also optimal for the ambient transport problem). We show that this occurs precisely when an optimal one-dimensional Kantorovich--Rubinstein (KR) potential pulls back through the slicer to an optimal KR potential for the ambient problem.
 
\paragraph{The latent duality, and what remains open}
At the level of duality, a connection between sliced and ambient OT is in
principle available. \underline{The Kantorovich--Rubinstein theorem writes $W_1$ as a
supremum over $1$-Lipschitz potentials $g:\mathbb{R}^d \to \mathbb{R}$, and a
scalar slice of $W_1$ with respect to a 1-Lipschitz slicer $f:\mathbb{R}^d\to\mathbb{R}$, admits its own KR representation over $1$-Lipschitz
functions $h:\mathbb{R}\to\mathbb{R}$. Because $h \circ f$ is $1$-Lipschitz whenever both factors are, scalar slicing with a $1$-Lipschitz slicer $f$ is exactly the ambient KR dual restricted to potentials that \emph{factor} as
$g = h \circ f$.} This factorization is implicit in the standard derivation of
the $W_1$ dual formulation; it underlies recent sharp comparisons between sliced and
standard $1$-Wasserstein distances~\cite{carlier2025sharp}, and the
max-sliced $W_1$ value has been written explicitly as a joint supremum over a
direction and a $1$-Lipschitz scalar function~\cite{boedihardjo2026sharp}.
Equivalence results, however, are partial and in one case fragile: Paty and
Cuturi~\cite{paty2019subspace} proved strong equivalence of subspace-robust
$W_2$ and the classical $W_2$, whereas the analogous $p=1$ claim for max-sliced
$W_1$~\cite{bayraktar2021strong} was retracted in the authors'
erratum~\cite{bayraktar2024errata}, which shows the equivalence fails for
$d \ge 2$. So linear slicing does not in general recover $W_1$ even up to a
constant.

What the prior min-sliced constructions establish is that the upper bound can be
optimized and lifted into a usable plan; what they do not establish is when the
upper bound is \emph{tight}, or whether a single, identifiable slicer makes the
lower and upper objectives \emph{coincide exactly} with $W_1$. Their selection
criteria are also sensitive to the one-dimensional ordering: the slice value
function is discontinuous and depends on the induced
permutation\footnote{for instance, for empirical measures, the slicer orders the source and target support points by their scalar coordinates, thereby inducing permutations that determine the lifted monotone coupling; as the slicer varies, the ordering remains fixed until two coordinates coincide and, if they subsequently exchange order, the induced permutation, and hence the lifted matching, may change abruptly }~\cite{chapel2025dgswp}, so tightness arguments that route through a
particular monotone lift require an implicit uniqueness assumption. 

A generative-modeling perspective further clarifies the role of the slicer class. KR duality is the basis of WGANs, where a discriminator or critic is trained over a parameterized class of Lipschitz functions in order to approximate an optimal KR potential \cite{arjovsky2017wasserstein}. Max-sliced Wasserstein GANs pursue a related strategy, but replace the full critic class by one-dimensional Wasserstein comparisons along an optimized, typically linear, projection \cite{deshpande2019maxsliced}. Generalized sliced Wasserstein distances enlarge this idea by allowing nonlinear defining functions, and max-GSW optimizes over this richer class of generalized slicers \cite{kolouri2019generalized}. From the viewpoint developed here, these approaches differ mainly in the admissible class of scalar witnesses used to approximate the KR objective. A linear max-sliced loss restricts the witness to functions that factor through a linear map, while a generalized sliced loss allows factorizations $h\circ f$ through nonlinear slicers. From this viewpoint, our results identify the exact scalar witness that
closes the sliced sandwich. In this work, $L(f)$ is the sliced lower value induced by
$f$, while $U(f)$ is the upper value obtained by minimizing the
ambient cost among slice-optimal couplings, obtaining $L(f)\leq W_1\leq U(f)$. Precise definitions are given in Section \ref{sec:slice_sandwich} (\eqref{eq:def_Lf}, \eqref{eq:def_U_tilde}, and \eqref{eq:sandwich_exact_slicer}). When the slicer class contains an
optimal KR potential $\phi^*$, the identity outer function $h(t)=t$ suffices,
and the same slicer attains both values:
$L(\phi^*)=W_1(\mu,\nu)=U(\phi^*)$.
Thus WGAN critics, max-sliced losses, and generalized slicers can be interpreted
as different restricted routes toward the same optimal scalar witness.

\paragraph{Contributions} 
This paper closes that gap: building on the aforementioned constructions, we determine when the lower and upper bounds on $W_1$ associated with the same slicer are \emph{simultaneously} tight, and identify the slicers that achieve this exactness.
We do not claim to introduce the min-sliced upper
bound, the lifted transport plan, or the KR-duality link between sliced and
ambient OT; these appear, respectively, in~\cite{mahey2023swgg,chapel2025dgswp},
in~\cite{mahey2023swgg,liu2025expected,tanguy2025sliced}, and latently in the
classical theory \cite{villani2003topics,villani2008optimal,santambrogio2015optimal}. Our contribution is to show that the gap collapses
\emph{exactly}, and to identify the object that collapses it. 
Our results are established for arbitrary Borel probability measures $\mu,\nu\in\cP_1(\bbR^d)$, that is, with finite first moments. Concretely, we establish the following:

\begin{itemize}[leftmargin=*]
\item \textbf{A two-sided sandwich for $1$-Lipschitz slicers.} Let $C$ denote the ambient $W_1$ cost, i.e.,  $C(x,y)\coloneqq\|x-y\|_2$ for all $x,y\in \mathbb R^d$, where $\|\cdot\|_2$ denotes the standard Euclidean norm on $\mathbb R^d$.
For any $1$-Lipschitz slicer $f:\mathbb R^d\to \mathbb R$, the induced one-dimensional cost $D^{(f)}(x,y)\coloneqq |f(x)-f(y)|$ trivially 
satisfies
$D^{(f)} \le C$ pointwise, giving the bound
$\langle \gamma_f, D^{(f)} \rangle \le W_1(\mu,\nu) \le \langle \gamma_f, C
\rangle$ for every slice-optimal plan $\gamma_f$, which are the optimizers of the OT problem between $\mu$ and $\nu$ replacing the ambient cost $C$ by the sliced cost $D^{(f)}$ (see Lemma \ref{lem: pointwise bounds}). Optimizing each side over a
slicer class of 1-Lipschitz functions $f$ yields a max-sliced \emph{lower} objective $L^*$ and a min-sliced
\emph{upper} objective $U^*$ with $L^* \le W_1(\mu,\nu) \le
U^*$ (see \eqref{eq:global_sandwich}). This places the existing lower (max-sliced) and upper
(min-sliced~\cite{mahey2023swgg,chapel2025dgswp}) constructions in a single
two-sided frame.

\item \textbf{KR potentials are exact, simultaneous optimizers.}
We show that if the slicer class contains an optimal KR potential $\phi^*$, then
$L(\phi^*) = W_1(\mu,\nu) = U(\phi^*)$, so $\phi^*$ \emph{simultaneously}
maximizes the lower objective and minimizes the upper one, and the sandwich
closes: $L^* = W_1(\mu,\nu) = U^*$. In particular, the min-sliced
upper bound is not merely optimizable but \emph{tight}, and the maximizing
lower slicer and minimizing upper slicer can be taken to be the \emph{same}
function. The lower identity follows from the KR dual with $h(t)=t$; the
load-bearing direction is the upper side, where every \emph{ambient} optimal
plan is shown to be optimal for the sliced cost induced by $\phi^*$. See Proposition \ref{prop:KR_simultaneous_no_uniqueness}, Theorem \ref{thm: characterization} and Corollary \ref{coro: characterization}.

\item \textbf{No uniqueness assumption.}
Our upper objective is defined by a ``tie-broken (lexicographic) 
transport
problem'': among all slice-optimal couplings, select one of minimal ambient cost (see its precise definition in \eqref{eq:def_U_tilde}).
This removes the dependence on the one-dimensional ordering (i.e., the source of the
discontinuity and permutation-sensitivity noted for min-SWGG/min-GSWP~\cite{chapel2025dgswp}) and makes the upper identity hold in full generality, including when ties or atomic pushforwards lead to multiple slice-optimal couplings.

\item \textbf{Nonlinearity may be needed for tightness.} Theorem~\ref{thm: characterization} shows that exactness for a slicer $f\in\mathcal{F}_1$ is equivalent to the existence of a $1$-Lipschitz scalar function $h_f$ such that $h_f\circ f$ is an optimal KR potential. Thus, requiring $\phi^*\in\mathcal{F}_1$ is a convenient sufficient condition (corresponding to $h_{\phi^*}(t)=t$) but it is not necessary. For a linear slicer $f_v(x)=v^\top x$, exactness requires an optimal KR potential of the ridge form $h(v^\top x)$, which need not exist in dimension $d\ge2$. This is consistent with the failure of exact linear max-sliced/$W_1$ equivalence~\cite{bayraktar2024errata} and motivates the generalized nonlinear slicer setting of~\cite{kolouri2019generalized,chapel2025dgswp}. From a generative-modeling perspective, max-sliced and max-generalized-sliced losses may therefore be viewed as restricted KR critics, with Theorem~\ref{thm: characterization} and Corollary~\ref{coro: characterization} identifying the precise factorization property needed to close the sandwich.

\end{itemize}

\paragraph{Scope and limitations}
A sufficient condition for the exact identities is that $\mathcal{F}_1$ contain an optimal KR potential. More generally, exactness requires some $f\in\mathcal{F}_1$ through which an optimal KR potential factors as $h_f\circ f$, with $h_f:\bbR\to\bbR$ $1$-Lipschitz. For restricted classes, such as linear projections or a fixed parametric family (e.g., a neural-network architecture), $L^*$ remains a valid lower bound and $U^*$ remains a valid upper bound, but neither bound need be tight. This recovers the regime in which prior min-sliced and max-sliced surrogates operate. 
Our results are established for arbitrary Borel probability measures $\mu,\nu\in\cP_1(\bbR^d)$ with the Euclidean ground cost. The experiments specialize to finite empirical measures, for which the lower and tie-broken upper objectives can be evaluated using finite-dimensional optimal transport programs.

\section{Slice-Conditional Transport and Two-Sided Bounds for \texorpdfstring{$W_1$}{W1}}
\label{sec:slice_sandwich}

\paragraph{Ambient $W_1$ optimal transport} 
Let $\cP(\bbR^d)$ be the set of probability measures on $\bbR^d$, and let
$
\cP_1(\bbR^d)\coloneqq\left\{\omega\in\cP(\bbR^d):\ \int_{\bbR^d}\|x\|_2\,\dd\omega(x)<\infty\right\}
$
be those with finite first moment (here, $\|\cdot\|_2$ denotes the standard Euclidean norm on $\bbR$, $\|x\|_2=\sqrt{\sum x_i^2}$). Throughout, $\mu$ and $\nu$ will be two
probability measures on $\mathbb{R}^d$ with finite first moments, i.e.
$\mu,\nu\in\cP_1(\bbR^d)$. This is the standing assumption for all of the
theoretical results below; it guarantees that the ambient and pushed-forward
Wasserstein distances, and the dual integrals appearing throughout, are finite.
We consider the ambient $W_1$ cost function 
\begin{equation}\label{eq: ambient cost}
C(x,y) \;\coloneqq\; \|x-y\|_2
\end{equation}
for which the ambient $W_1$ distance is
\begin{equation}
W_1(\mu,\nu) \;\coloneqq\; \min_{\gamma\in\Gamma(\mu,\nu)} \langle \gamma,C\rangle,
\label{eq:ambient_W1_primal}
\end{equation}
where we denote by
$ 
\Gamma(\mu,\nu) \;\coloneqq\; \left\{ \gamma\in\cP(\bbR^{d\times d}) \, : \, \gamma(\cdot\times \bbR^d) =\mu,\; \gamma(\bbR^d\times \cdot) = \nu \right\}
$
the set of feasible couplings, and 
where $\langle \gamma,C\rangle$ denotes the pairing between $\gamma$ and $C$, i.e., $\langle \gamma,C\rangle := \int_{\bbR^d\times\bbR^d} C(x,y) \, \dd \gamma(x,y)$.
Let
\begin{equation}\label{eq: opt ambient plans}
\Gamma_C^*
\;\coloneqq\;
\argmin_{\gamma\in\Gamma(\mu,\nu)}
\langle \gamma,C\rangle
\end{equation}
denote the set of ambient optimal transport plans.
We recall that \eqref{eq:ambient_W1_primal} is finite 
by the standing assumption $\mu,\nu\in\cP_1(\bbR^d)$.

\paragraph{Slice cost and slice-optimal plans}
Given a slicer $f:\mathbb{R}^d\to\mathbb{R}$, we replace the ambient cost \eqref{eq: ambient cost} by defining the induced one-dimensional cost function 
\begin{equation}\label{eq: projected cost}
D^{(f)}(x,y) \;\coloneqq\; |f(x)-f(y)|.
\end{equation}
In analogy with \eqref{eq: opt ambient plans}, the set of slice-optimal ambient couplings is now
\begin{equation}
\Gamma_f
\;\coloneqq\;
\displaystyle\argmin_{\gamma\in\Gamma(\mu,\nu)}
\langle \gamma,D^{(f)}\rangle.
\label{eq:def_Gamma_f}
\end{equation}
Equivalently, the optimal value of \eqref{eq:def_Gamma_f} is the one-dimensional Wasserstein distance
between the pushed-forward measures $f_{\#}\mu$ and $f_{\#}\nu$ (where for any $\omega\in\cP(\bbR^d)$ and any $\omega$-measurable function $g:\bbR^d\to\bbR$, $g_{\#}\omega(A) \;\coloneqq\; \omega(\{x\in\mathbb R^d:\, g(x)\in A\})$ for all Borel measurable set $A\subseteq \mathbb R$).
Indeed, since $f$ is $1$-Lipschitz,  $f_\#\mu,f_\#\nu\in\mathcal{P}_1(\bbR)$, and when considering the cost \eqref{eq: projected cost}, we have  
\begin{equation}
\min_{\gamma\in\Gamma(\mu,\nu)}
\langle \gamma,D^{(f)}\rangle
=
W_1(f_\#\mu,f_\#\nu).
\label{eq:slice_value_pushforward}
\end{equation}
Pushing any
$\gamma\in\Gamma(\mu,\nu)$ forward through $(f,f)$ produces a coupling of $f_\#\mu$
and $f_\#\nu$ with the same cost, which gives one inequality. Conversely, since
$\bbR^d$ and $\bbR$ are standard Borel spaces, disintegrating $\mu$ and $\nu$
along $f$ allows any coupling of $f_\#\mu$ and $f_\#\nu$ to be lifted to an
element of $\Gamma(\mu,\nu)$, giving the reverse inequality.

\paragraph{A sandwich inequality for 1-Lipschitz slicers}
Assume that $f:\bbR^d\to\bbR$ is $1$-Lipschitz, that is, $|f(x)-f(y)|
\;\le\;
\|x-y\|_2$ for all
$x,y\in\mathbb{R}^d$.
Then
\begin{equation}
D^{(f)}(x,y) = |f(x)-f(y)| \le \|x-y\|_2 = C(x,y).
\label{eq:D_le_C}
\end{equation}

\begin{lemma}\label{lem: pointwise bounds}
    Let $f:\mathbb R^d\to \mathbb R$ be a $1$-Lipschitz map. For any $\gamma_f\in\Gamma_f$, we have the two-sided bound
\begin{equation}
\langle \gamma_f,D^{(f)}\rangle
\;\le\;
W_1(\mu,\nu)
\;\le\;
\langle \gamma_f,C\rangle.
\label{eq:sandwich_any_slice_plan}
\end{equation}
\end{lemma}

\begin{proof}
Let $\gamma^*\in\Gamma_C^*$. Since $\gamma_f$ is optimal for the slice cost, $D^{(f)}\le C$ pointwise due to \eqref{eq:D_le_C}, and $\gamma^*$ is a nonnegative measure, we have
\begin{equation*}
\langle \gamma_f,D^{(f)}\rangle\leq \langle \gamma^*,D^{(f)}\rangle
\le
\langle \gamma^*,C\rangle
=
W_1(\mu,\nu).
\end{equation*}
This proves the left inequality in \eqref{eq:sandwich_any_slice_plan}. The right inequality in \eqref{eq:sandwich_any_slice_plan} follows because every $\gamma_f\in\Gamma_f$
is feasible for the ambient problem, that is, for \eqref{eq:ambient_W1_primal}.
\end{proof}

\paragraph{Lower and upper sliced objectives}
Let $\mathcal{F}_1$ be a chosen class of $1$-Lipschitz slicers
$f:\mathbb{R}^d\to\mathbb{R}$. For $f\in\mathcal{F}_1$, define the lower slice value as \eqref{eq:slice_value_pushforward}:
\begin{equation}
L(f) \;\coloneqq\; \min_{\gamma\in\Gamma(\mu,\nu)} \langle \gamma,D^{(f)}\rangle = W_1(f_\#\mu,f_\#\nu).
\label{eq:def_Lf}
\end{equation}
Then the max-sliced lower objective is
\begin{equation*}
L^*
\;\coloneqq\;
\sup_{f\in\mathcal{F}_1} L(f).
\label{eq:Lstar}
\end{equation*}
Note that one can rewrite \eqref{eq:def_Gamma_f} as
\begin{equation}\label{eq: Gamma_f equivalent}
\Gamma_f=\{\gamma\in \Gamma(\mu,\nu): \, \langle \gamma, D^{(f)}\rangle=L(f)\},
\end{equation}
and by \eqref{eq:sandwich_any_slice_plan}, for every $\gamma\in \Gamma_f$,
\begin{equation}\label{eq:L_lower_pointwise}
    L(f)= \langle \gamma, D^{(f)}\rangle\leq W_1(\mu,\nu) \qquad \forall f\in \mathcal{F}_1,
\end{equation}
obtaining
\begin{equation}
L^*
\le
W_1(\mu,\nu).
\label{eq:Lstar_lower}
\end{equation}

The upper side requires more care. If the slice problem has multiple optimal couplings, different
elements of $\Gamma_f$ can have different ambient costs $\langle\gamma,C\rangle$.
Using \eqref{eq: Gamma_f equivalent}, instead of
choosing an arbitrary slice-optimal plan, we define the canonical upper evaluation as the best ambient
cost among slice-optimal couplings:
\begin{equation}
U(f) \;\coloneqq\; \min_{\gamma\in\Gamma_f}
\langle \gamma,C\rangle =
\min_{\gamma\in \displaystyle \underset{\eta\in\Gamma(\mu,\nu)}{\argmin}\langle \eta,D^{(f)}\rangle}\langle \gamma,C\rangle.
\label{eq:def_U_tilde}
\end{equation}
This is a two-stage, lexicographically ordered optimal transport problem, where the sliced cost has priority over the ambient cost. The sliced cost is the primary objective, meaning that one first restricts to the set $\Gamma_f$ of slice-optimal couplings. Then, one minimizes the ambient cost over all slice-cost minimizers. Thus, the ambient cost acts as a secondary tie-breaking criterion among slice-optimal couplings.

We notice that the set $\Gamma(\mu,\nu)$ is nonempty and weakly compact. Moreover, since $f$ is continuous, the sliced cost
$D^{(f)}$ is continuous and nonnegative, so $\Gamma_f$
is nonempty and weakly compact. The ambient cost $C$ is nonnegative and lower
semicontinuous, so its minimum in \eqref{eq:def_U_tilde} over $\Gamma_f$ is attained and, therefore,
it is well defined.

The corresponding min-sliced upper objective is
\begin{equation*}
U^*
\;\coloneqq\;
\inf_{f\in\mathcal{F}_1}
U(f).
\label{eq:Ustar_tilde}
\end{equation*}
Taking the minimum over $\gamma\in\Gamma_f$ in the right inequality given in \eqref{eq:sandwich_any_slice_plan}, yields
\begin{equation}
W_1(\mu,\nu)
\le
U(f)
\qquad
\forall f\in\mathcal{F}_1.
\label{eq:U_tilde_upper_pointwise}
\end{equation}
Therefore
\begin{equation}
W_1(\mu,\nu)
\le
U^*.
\label{eq:Ustar_tilde_upper}
\end{equation}
Finally, from \eqref{eq:L_lower_pointwise} and \eqref{eq:U_tilde_upper_pointwise} we have
\begin{equation}
     L(f)\leq W_1(\mu,\nu)\leq U(f)\qquad
\forall f\in\mathcal{F}_1, \label{eq:sandwich_exact_slicer}
\end{equation}
and combining \eqref{eq:Lstar_lower} and \eqref{eq:Ustar_tilde_upper} yields
\begin{equation}
L^*
\le
W_1(\mu,\nu)
\le
U^*.
\label{eq:global_sandwich}
\end{equation}

\begin{remark}
In one dimension, optimal couplings for the cost $|s-t|$ may be chosen to be monotone (canonical one-dimensional representatives). For discrete
measures, this can be realized by sorting the support points and applying the usual greedy monotone
matching of cumulative mass distributions. 
This coupling is uniquely determined by the one-dimensional marginals, although it need not be the unique optimal coupling for the one-dimensional Wasserstein distance $W_1$: since $r\mapsto|r|$ is not strictly convex, nonmonotone optimal couplings may also exist. 
In addition, if distinct points in the original supports in $\mathbb R^d$ have the same one-dimensional projection, lifting the projected coupling back to the original supports may introduce additional nonuniqueness.
The
definition of $U(f)$ avoids this ambiguity by selecting, among all ambient couplings that
are optimal for the sliced problem, one with minimal ambient cost. Equivalently, one may view
\eqref{eq:def_U_tilde} as a canonical tie-breaking rule for lifting a sliced optimal plan back to the
ambient supports.
\end{remark}

\begin{remark}
    Notice that the upper bound $U^*$ is easy to attain.  For instance, any constant function $f$ induces a cost $D^{(f)}$ that is identically zero, obtaining $\Gamma_f = \Gamma( \mu, \nu )$, and therefore  $U(f)=W_1(\mu,\nu)=U^*$ (if $\mathcal{F}_1$ contains a constant slicer). See also Example \ref{example: tie-breaking} below.  In general, every 1-Lipschitz slicer $f\in \mathcal F_1$ satisfying $\Gamma_C^*\cap \Gamma_f\neq\emptyset$ attains the upper bound $U(f)=W_1(\mu,\nu)=U^*$. 
    The purpose of $U$ is structural rather than computational: it tests whether the slice-optimal set of couplings $\Gamma_f$ contains an ambient-optimal plan and provides an ordering-independent upper value, but it is not proposed as a cheaper general method for computing $W_1$.
\end{remark}

    Subsequently, we will show that if $\phi^*$ is an optimal Kantorovich--Rubinstein potential from the dual formulation of $W_1(\mu,\nu)$, it attains \emph{both} upper and lower bounds, closing the gap (see Proposition \ref{prop:KR_simultaneous_no_uniqueness}, Theorem \ref{thm: characterization} and Corollary \ref{coro: characterization}).

\paragraph{KR duality and scalar slicing}
For $\mu,\nu\in\cP_1(\bbR^d)$, the Kantorovich--Rubinstein dual formulation of the ambient $W_1$ distance is
\begin{equation}\label{eq: KR}
    W_1(\mu,\nu) = \sup_{\substack{g:\mathbb{R}^d\to\mathbb{R}\\ \Lip(g)\le 1}} \left[ \int_{\bbR^d} g(x) \, \dd \mu(x) - \int_{\bbR^d} g(y) \, \dd \nu(y) \right]. 
\end{equation}
For a fixed slicer $f$, the one-dimensional sliced value also admits the KR dual representation
\begin{equation} \label{eq:sliced_KR_dual}
L(f)  = W_1(f_\#\mu,f_\#\nu) = \sup_{\substack{h:\mathbb{R}\to\mathbb{R}\\ \Lip(h)\le 1}} J_{\mathrm{KR}}(f), \ \text{where } J_{\mathrm{KR}}(f) :=  \int_{\bbR} h(s) \, \dd f_{\#}\mu(s) - \int_{\bbR} h(t) \, \dd f_{\#}\nu(t).
\end{equation}
Thus scalar slicing restricts the ambient KR dual class to potentials of the form $g=h\circ f$ for 
$f:\mathbb{R}^d\to\mathbb{R}$,
\qquad
$h:\mathbb{R}\to\mathbb{R}$ with 
$\Lip(f)\le 1$ and $\Lip(h)\le 1$,
where $\Lip(f)$ (resp. $\Lip(h)$) denotes the Lipschitz constant associated to $f$ (resp. to $h$)\footnote{i.e., $|f(x)-f(y)|\leq \Lip(f)\|x-y\|_2$ for all $x,y\in \mathbb R^d$; and $|h(u)-h(v)|\leq \Lip(h)|u-v|$ for all $u,v\in \mathbb R$}.
Indeed, $h\circ f$ is $1$-Lipschitz whenever both $h$ and $f$ are $1$-Lipschitz. Therefore,
\eqref{eq:sliced_KR_dual} is precisely the ambient KR dual restricted to potentials that factor
through the scalar slicer $f$.

Finally, we recall that for both \eqref{eq: KR} and \eqref{eq:sliced_KR_dual}, the supremum is attained
\cite[Theorem~5.10(iii) and Particular Cases~5.4 and~5.16]
{villani2008optimal}. 

\paragraph{KR potentials as optimal slicers}
Let $\phi^*:\mathbb{R}^d\to\mathbb{R}$ be an optimal KR potential for
$W_1(\mu,\nu)$:
\begin{equation}
\Lip(\phi^*)\le 1, \qquad \int_{\bbR^d} \phi^*(x) \, \dd \mu(x) - \int_{\bbR^d} \phi^*(y) \, \dd \nu(y) = W_1(\mu,\nu).
\label{eq:KR_potential}
\end{equation}

\begin{lemma}\label{lem: lower_bound_attained}
    Let $\phi^*:\mathbb R^d\to \mathbb R$ be an optimal KR potential satisfying
\eqref{eq:KR_potential}. Then $L(\phi^*)
=
W_1(\mu,\nu)$.
\end{lemma}

\begin{proof}
    Choosing the one-dimensional dual function $h(t)=t$ in
\eqref{eq:sliced_KR_dual} gives
\begin{equation}
L(\phi^*) = W_1((\phi^*)_\#\mu,(\phi^*)_\#\nu) \ge
\int_{\bbR^d} \phi^*(x) \, \dd \mu(x) - \int_{\bbR^d} \phi^*(y) \, \dd \nu(y) 
= W_1(\mu,\nu).
\label{eq:L_phi_lower_from_KR}
\end{equation}
On the other hand, since $\phi^*$ is $1$-Lipschitz, \eqref{eq:Lstar_lower} implies
\begin{equation}
L(\phi^*)
\le
W_1(\mu,\nu).
\label{eq:L_phi_upper_from_sandwich}
\end{equation}
Combining \eqref{eq:L_phi_lower_from_KR} and \eqref{eq:L_phi_upper_from_sandwich} yields $L(\phi^*)
=
W_1(\mu,\nu)$.

We can also see this directly at the level of couplings: For every
$\gamma\in\Gamma(\mu,\nu)$,
\begin{align}
\langle \gamma,D^{(\phi^*)}\rangle & =
 \int_{\bbR^d\times \bbR^d} \left|\phi^*(x)-\phi^*(y)\right| \, \dd \gamma(x,y) \nonumber\\
 & \ge \int_{\bbR^d\times \bbR^d} \left(\phi^*(x)-\phi^*(y)\right) \, \dd \gamma(x,y) \nonumber\\
 & = \int_{\bbR^d} \phi^*(x) \, \dd \mu(x) - \int_{\bbR^d} \phi^*(y) \, \dd \nu(y) \nonumber\\
 & = W_1(\mu,\nu).
\label{eq:every_coupling_slice_lower}
\end{align}
Taking the minimum over $\gamma\in\Gamma(\mu,\nu)$ gives
$L(\phi^*)\ge W_1(\mu,\nu)$, and the reverse inequality follows from
$\phi^*$ being $1$-Lipschitz (see \eqref{eq:Lstar_lower}).
\end{proof}

The next result shows that KR potentials for $W_1(\mu,\nu)$ are, moreover, \emph{simultaneous} lower and upper optimizers for \eqref{eq:sandwich_exact_slicer} reaching both equalities in \eqref{eq:global_sandwich}.

\begin{proposition}
\label{prop:KR_simultaneous_no_uniqueness}
Assume $\phi^*\in\mathcal{F}_1$ is an optimal KR potential satisfying
\eqref{eq:KR_potential}. Then
\begin{equation}
L(\phi^*)
=
W_1(\mu,\nu)
=
U(\phi^*).
\label{eq:tight_bounds_phi_tilde}
\end{equation}
Consequently, $L^*
=
W_1(\mu,\nu)
=
U^*$,
and $\phi^*$ simultaneously maximizes the lower objective and minimizes the tie-broken upper
objective:
\begin{equation}
\phi^*
\in
\argmax_{f\in\mathcal{F}_1} L(f),
\qquad
\phi^*
\in
\argmin_{f\in\mathcal{F}_1} U(f).
\end{equation}
\end{proposition}

\begin{proof}
Lemma \ref{lem: lower_bound_attained} already shows that $L(\phi^*)=\min_{\gamma\in\Gamma(\mu,\nu)}
\langle \gamma,D^{(\phi^*)}\rangle=W_1(\mu,\nu)$.
Moreover, let $\gamma^*\in\Gamma_C^*$ be any ambient optimal transport plan. Since $\phi^*$ is
$1$-Lipschitz, $D^{(\phi^*)}\le C$ pointwise. Hence
\begin{equation}
L(\phi^*) \leq \langle \gamma^*,D^{(\phi^*)}\rangle
\le
\langle \gamma^*,C\rangle
=
W_1(\mu,\nu)=L(\phi^*)=W_1(\mu,\nu)=\langle \gamma^*,C\rangle,
\label{eq:gamma_star_slice_upper}
\end{equation}
implying that every inequality in \eqref{eq:gamma_star_slice_upper} is, in fact, an equality. 
Therefore, $\gamma^*\in\Gamma_{\phi^*}$; that is, every ambient optimal plan is also optimal for the
slice cost induced by the KR potential.

By the definition of $U$,
\begin{equation*}
U(\phi^*)
=
\min_{\gamma\in\Gamma_{\phi^*}}
\langle \gamma,C\rangle
\le
\langle \gamma^*,C\rangle
=
W_1(\mu,\nu).
\label{eq:U_tilde_phi_upper}
\end{equation*}
On the other hand, by \eqref{eq:U_tilde_upper_pointwise}, $W_1(\mu,\nu)
\le
U(\phi^*)$. Thus $W_1(\mu,\nu)
=
U(\phi^*)$. 
Together with $L(\phi^*)=W_1(\mu,\nu)$, this proves
\eqref{eq:tight_bounds_phi_tilde}.

Finally, since $L(f)\le W_1(\mu,\nu)$ for all $f\in\mathcal{F}_1$, the equality
$L(\phi^*)=W_1(\mu,\nu)$ implies
$L^*=W_1(\mu,\nu)$, and 
$\phi^*\in\argmax_{f\in\mathcal{F}_1}L(f)$.
Similarly, since $U(f)\ge W_1(\mu,\nu)$ for all $f\in\mathcal{F}_1$, the equality
$U(\phi^*)=W_1(\mu,\nu)$ implies
$U^*=W_1(\mu,\nu)$, and $
\phi^*\in\argmin_{f\in\mathcal{F}_1}U(f)$.
\end{proof}

\begin{remark}[Role of the slicer class]
\label{rem:slicer_class}
The exact identities in Proposition~\ref{prop:KR_simultaneous_no_uniqueness} require that the chosen
slicer class $\mathcal{F}_1$ contains at least one optimal KR potential $\phi^*$. If
$\mathcal{F}_1$ is restricted, for example to linear projections or to a fixed neural-network architecture,
then $L^*$ remains a valid lower bound and $U^*$ remains a valid upper bound, but the
bounds need not be tight. Moreover, in Theorem \ref{thm: characterization} we will provide a full characterization of $\mathcal{F}_1$-slicers that attain both the lower ($L^*$) and upper ($U^*$) bounds: optimal KR potentials  not only reach both bounds, as shown in Proposition \ref{prop:KR_simultaneous_no_uniqueness}, but also if a slicer $f\in \mathcal{F}_1$ closes the gap (i.e. $L(f)=U(f)$), then an optimal KR potential $\phi^*$ as in \eqref{eq:KR_potential} factors through $f$ (see Theorem \ref{thm: characterization} part \textit{(iii)} for the precise description).
\end{remark}

\begin{remark}[Why uniqueness is not needed] Note that no uniqueness assumption on the slice-optimal coupling is required in the proof of Proposition \ref{prop:KR_simultaneous_no_uniqueness}.
The tie-broken upper objective $U$ removes this requirement. The proof only needs
the existence of one ambient optimal plan $\gamma^*$ that is also slice-optimal for
$D^{(\phi^*)}$. Once such a plan exists, the secondary minimization in $U(\phi^*)$ can only
decrease the ambient cost relative to $\langle\gamma^*,C\rangle=W_1(\mu,\nu)$, while feasibility
forces it to be no smaller than $W_1(\mu,\nu)$. Hence equality follows without any uniqueness
assumption.
\end{remark}

\begin{example}[Importance of tie-breaking]\label{example: tie-breaking}
Let $\theta_i=2\pi(i-1)/n$ and $z_i=(\cos\theta_i,\sin\theta_i)\in\mathbb S^1$, for $i=1,\ldots,n$.
For $0<r<R$, consider the uniform empirical measures $\mu=\frac1n\sum_{i=1}^n\delta_{R z_i}$  and 
$\nu=\frac1n\sum_{i=1}^n\delta_{r z_i}$,
with, for example, $r=1$ and $R=2$. Fix the radial slicer
$f(x)=\|x\|_2$ (which, at the same time, is an optimal KR potential). Then, for every $i,j$,
\begin{equation*}
    D^{(f)}(rz_i,Rz_j)
=
\bigl|f(rz_i)-f(Rz_j)\bigr|
=
R-r.    
\end{equation*}
Consequently, every feasible coupling has the same sliced cost,
and hence $\Gamma_f=\Gamma(\mu,\nu)$. This illustrates that for a fixed slicer $f$, the set $\Gamma_f$ defined in \eqref{eq:def_Gamma_f}
might contain many different couplings.
Moreover, the reverse triangle inequality gives $\|rz_i-Rz_j\|_2\geq R-r  $,
while the radial coupling $rz_i\mapsto Rz_i$ attains this lower
bound. Therefore,
\begin{equation*}
    W_1(\mu,\nu)=R-r
\qquad\text{and}\qquad
U(f)=R-r.    
\end{equation*}
For a  permutation coupling $\pi$, its ambient cost is
\begin{equation*}
    A_\pi
:=
\langle \pi,C\rangle
=
\frac1n\sum_{i=1}^n
\|rz_i-Rz_{\pi(i)}\|_2
=
\frac1n\sum_{i=1}^n
\sqrt{
r^2+R^2
-2rR\cos(\theta_i-\theta_{\pi(i)})
}.    
\end{equation*}
Although every $\pi$ is sliced-optimal, the values $A_\pi$
can differ substantially. In particular, if $\pi$ is induced by $Rz_i\mapsto rz_i$ for all $i\in\{1,\dots,n\}$, i.e., the identity map at the level of the indices $i=1,\dots, n$, then    $A_{\pi}=R-r$;
whereas, for even $n$, the antipodal permutation ($Rz_i\mapsto rz_{i+n/2}=-rz_i$ for all $i\in\{1,\dots,n\}$)
satisfies $A_{\pi}=R+r$. In Figure \ref{fig:tie-break-simple}, we plot a histogram of $A_\pi$ over randomly sampled permutations illustrating that an arbitrary slice-optimal lift can have a
large ambient cost, whereas the tie-broken objective $U(f)$ selects
the smallest ambient cost among all slice-optimal couplings (reaching, in this case, $W_1(\mu,\nu)=U(f)=R-r$).  

\begin{figure}[t!]
\centering
\includegraphics[width=\linewidth]{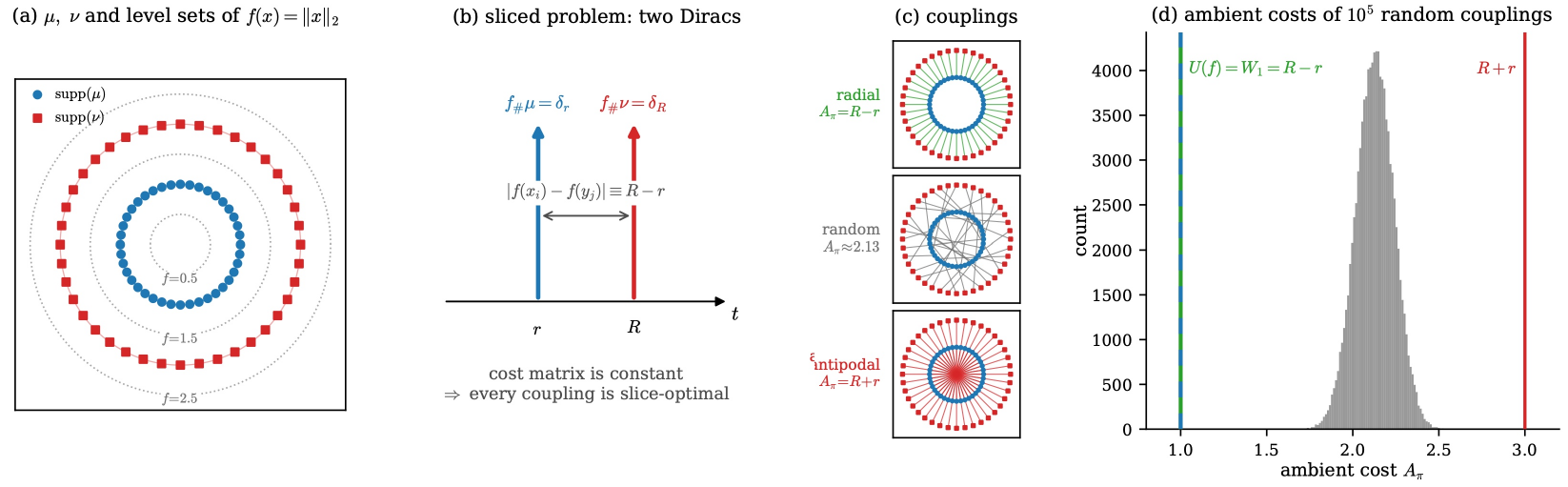}
\vspace{-0.25in}
\caption{Tie-breaking on concentric circles ($n=40$, $r=1$, $R=2$).
(a)~Uniform empirical measures $\mu,\nu$ with level sets of the radial
slicer $f(x)=\|x\|_2$. (b)~Both pushforwards are Dirac deltas, so every entry
of the sliced cost matrix equals $R-r$ and \emph{every} feasible coupling
is slice-optimal. (c)~Three slice-optimal permutation couplings whose
ambient costs $A_\pi$ nevertheless range from $R-r$ (radial) to $R+r$
(antipodal). (d)~Ambient costs of $100{,}000$ randomly sampled
permutation couplings. The two-stage optimization defining $U(f)$
selects the slice-optimal coupling of minimum ambient cost, so the
green $W_1$ line and blue $U(f)$ line coincide at $R-r$, while the red
line marks $R+r$.}
\label{fig:tie-break-simple}
\end{figure}
\end{example}

\begin{example}[Scaling path]\label{example: scaling path}   
Let $\phi^*$ be an optimal KR potential for
$W_1(\mu,\nu)$, and define
$f_\lambda\coloneqq\lambda\phi^*$, $\lambda\in[0,1]$.  
Notice that $\operatorname{Lip}(f_\lambda)
=
\lambda\operatorname{Lip}(\phi^*)
\leq \lambda$
so every $\lambda\in[0,1]$ satisfies the $1$-Lipschitz constraint.
If $\operatorname{Lip}(\phi^*)=1$ (such as the radial potential $\phi^*(x)=\|x\|_2$), then $\lambda>1$ lies outside the admissible slicer class.
For every $\lambda>0$, $D^{(f_\lambda)}
=
\lambda D^{(\phi^*)}$,
and hence multiplying the sliced cost by the positive scalar
$\lambda$ does not change its minimizers
(i.e., $\Gamma_{f_\lambda}=\Gamma_{\phi^*}$).
Consequently,
\begin{equation*}
    L(f_\lambda)
=
\lambda L(\phi^*)
=
\lambda W_1(\mu,\nu),
\qquad
U(f_\lambda)
=
U(\phi^*)
=
W_1(\mu,\nu).    
\end{equation*}
At $\lambda=0$, the slicer is constant and
$D^{(f_0)}\equiv0$, so that $\Gamma_{f_0}=\Gamma(\mu,\nu)$.
Therefore,
$L(f_0)=0$,
$U(f_0)
=
W_1(\mu,\nu)$.
Thus, along the entire scaling path,
\begin{equation*}
    \frac{L(f_\lambda)}{W_1(\mu,\nu)}=\lambda,
\qquad
\frac{U(f_\lambda)}{W_1(\mu,\nu)}=1,
\qquad
\frac{U(f_\lambda)-L(f_\lambda)}{W_1(\mu,\nu)}
=1-\lambda.
\end{equation*}
We consider again the concentric-circle configuration from Example \ref{example: tie-breaking}, where $\mu$ and $\nu$ are uniform empirical measures supported on the
outer circle of radius $R$ and on the inner circle of radius
$r<R$, respectively, for which $\phi^*(x)=\|x\|_2$ and $
W_1(\mu,\nu)=R-r$. In Figure \ref{fig: scaling-path},
we compare the theoretical curves
$\lambda\mapsto\lambda$ and $\lambda\mapsto1$ with the computed
quantities
\begin{equation*}
    \lambda\longmapsto
\frac{L(f_\lambda)}{W_1(\mu,\nu)}
\qquad\text{and}\qquad
\lambda\longmapsto
\frac{U(f_\lambda)}{W_1(\mu,\nu)}.    
\end{equation*}
This example shows that the equality $U(f)=W_1(\mu,\nu)$ alone
does not imply that $f$ is an informative or exact slicer (in fact, such equality holds
even for the collapsed slicer $f_0\equiv0$). The meaningful condition
is instead closure of the sandwich,
$L(f)=U(f)$. 

\begin{figure}[t!]
    \centering
    \includegraphics[width=0.9\linewidth]{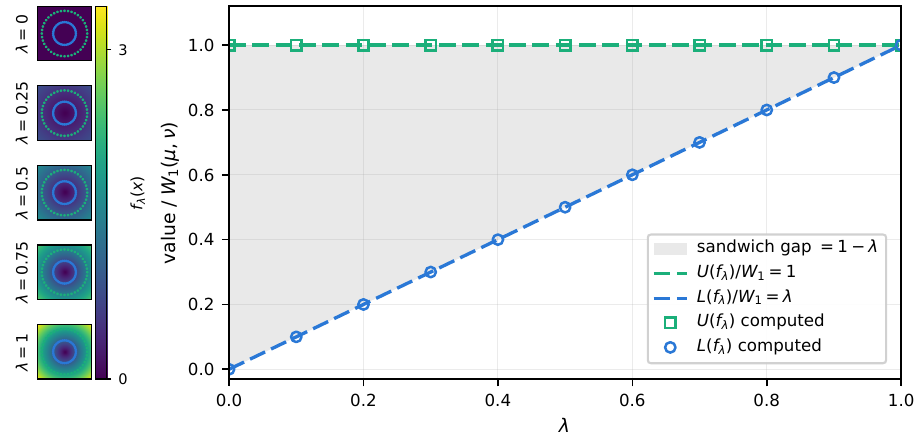}
    \vspace{-0.2in}\caption{Scaling path for the concentric-circle example of empirical measures with $n=40$ equally
spaced support points on circles of radii $r=1$ and $R=2$. We consider
$f_\lambda=\lambda\phi^*$, where
$\phi^*(x)=\|x\|_2$ and $\lambda\in[0,1]$. Left: $f_\lambda$ for
$\lambda\in\{0,0.25,0.5,0.75,1\}$ on a common colour scale, with the two
circles overlaid. Right: dashed lines show the
theoretical identities, while markers show the values computed using
discrete optimal transport and the exact two-stage definition of
$U(f_\lambda)$. Along the path,
$L(f_\lambda)/W_1(\mu,\nu)=\lambda$ and
$U(f_\lambda)/W_1(\mu,\nu)=1$, so the normalized sandwich gap equals
$1-\lambda$. In particular, the collapsed slicer $f_0\equiv0$ already
satisfies $U(f_0)=W_1(\mu,\nu)$ despite having $L(f_0)=0$; the
sandwich closes only at $\lambda=1$.}
    \label{fig: scaling-path}
\end{figure}
\end{example}

The two examples \ref{example: tie-breaking} and \ref{example: scaling path} isolate the distinct roles of the two sides of the sandwich. The first (Example \ref{example: tie-breaking}) shows that the tie-broken definition of $U$ is essential as many couplings can be equally optimal after slicing while having very different ambient transport costs. The second (Example \ref{example: scaling path}) shows that the equality $U(f)=W_1(\mu,\nu)$ alone does not characterize an informative slicer. This motivates taking the closure $L(f)=U(f)$ as the relevant notion of exactness. The following result addresses precisely when this occurs. Indeed, 
the next theorem gives an exact-slicer characterization: It identifies the $1$-Lipschitz slicers for which the lower--upper interval collapses, namely
$L(f)=W_1(\mu,\nu)=U(f)$.
The equivalences show that this exactness can be detected either through the sliced lower value $L(f)=W_1(\mu,\nu)$, the vanishing of the gap $L(f)=U(f)$, a factorization of an optimal KR potential, or a calibration property of slice-optimal and ambient-optimal plans.

\begin{theorem}\label{thm: characterization}
Let $f:\mathbb R^d\to\mathbb R$ be $1$-Lipschitz and  $S_f \coloneqq \{(x,y): \, D^{(f)}(x,y)=C(x,y)\}$.
Then the following statements are equivalent:
\begin{enumerate}
    \item[(i)] $L(f)=W_1(\mu,\nu)$.
    \item[(ii)] $L(f)=U(f)$.
    \item[(iii)] 
    There exists a function
$h_f:\mathbb R\to\mathbb R$ with $\Lip(h_f)\leq 1$
such that $h_f\circ f$ is an optimal KR
potential for $W_1(\mu,\nu)$.
   \item[(iv)] There exists $\gamma_f\in\Gamma_f$ with $C=D^{(f)}$ $\gamma_f$-almost everywhere; equivalently, since $S_f$ is closed, $\gamma_f$ is supported on $S_f$.
    \item[(v)] Every $\gamma^*\in \Gamma_C^*$ belongs to $\Gamma_f$ and satisfies $C=D^{(f)}$ $\gamma^*$-almost everywhere.
\end{enumerate}
\end{theorem}

\begin{proof}
First assume (i), namely $L(f)=W_1(\mu,\nu)$. Let $\gamma^*\in\Gamma_C^*$. Similarly as in \eqref{eq:gamma_star_slice_upper},
\begin{equation*}
L(f)
\leq
\langle \gamma^*,D^{(f)}\rangle
\leq
\langle \gamma^*,C\rangle
=
W_1(\mu,\nu)
=
L(f).
\end{equation*}
Hence $\langle \gamma^*,D^{(f)}\rangle=L(f)$, which implies $\gamma^*\in\Gamma_f$. Therefore
\begin{equation*}
U(f)
=
\min_{\gamma\in\Gamma_f}\langle\gamma,C\rangle
\leq
\langle \gamma^*,C\rangle
=
W_1(\mu,\nu).
\end{equation*}
Consequently, $L(f)=W_1(\mu,\nu)=U(f)$. Thus (i) implies (ii).
Conversely, if (ii) holds, then \eqref{eq:sandwich_exact_slicer} gives 
\begin{equation*}
L(f)\leq W_1(\mu,\nu)\leq U(f)=L(f),
\end{equation*}
so $L(f)=W_1(\mu,\nu)$. Thus (ii) implies (i). Therefore (i) and (ii) are equivalent.

Next we prove the equivalence with (iii). Suppose first that (i) holds. Since, by definition \eqref{eq:def_Lf}, $L(f)=W_1(f_\#\mu,f_\#\nu)$,
the one-dimensional KR duality  gives a $1$-Lipschitz function
$h_f:\mathbb R\to\mathbb R$ such that
\begin{equation*}
\int_{\bbR^d} h_f\circ f \,\dd \mu-\int_{\bbR^d} h\circ f \,\dd \nu=\int_{\bbR} h_f \,\dd f_{\#}\mu-\int_{\bbR} h_f \,\dd f_{\#}\nu
=
W_1(f_\#\mu,f_\#\nu)
=
L(f)
\end{equation*}
(since the supremum in \eqref{eq:sliced_KR_dual} is attained).
By (i), this equals $W_1(\mu,\nu)$. Since both $h_f$ and $f$ are $1$-Lipschitz, the composition
$h_f\circ f$ is $1$-Lipschitz on $\mathbb R^d$. Hence $h_f\circ f$ is feasible for the ambient
KR dual problem \eqref{eq: KR} and attains the value $W_1(\mu,\nu)$. Therefore $h_f\circ f$ is an
optimal KR potential for $W_1(\mu,\nu)$, proving (iii).

Conversely, suppose (iii) holds. Then, there exists $h_f:\bbR\to\bbR$, which is $1$-Lipschitz, and therefore it is feasible for the one-dimensional KR dual between $f_\#\mu$ and $f_\#\nu$, such that $h_f\circ f$ is an optimal ambient KR potential. So, 
\begin{equation*} 
W_1(\mu,\nu) =\int_{\bbR^d} h_f\circ f \,\dd \mu - \int_{\bbR^d} h_f\circ f \, \dd \nu=\int_{\bbR} h_f \,\dd f_{\#}\mu - \int_{\bbR} h_f \, \dd f_{\#}\nu 
\leq W_1(f_\#\mu,f_\#\nu) = L(f). 
\end{equation*} Thus $W_1(\mu,\nu)\leq L(f)$. Combining this with \eqref{eq:sandwich_exact_slicer}, we obtain $L(f)=W_1(\mu,\nu)$, which is (i).

Now, we prove the equivalence with (iv) and (v).
First, (v) implies, in particular, (iv). Now, 
suppose (i) holds, and let
$\gamma^*\in\Gamma_C^*$. As shown above, in the part (i)$\Rightarrow$(ii),  $\gamma^*\in\Gamma_f$. Moreover,
\begin{equation*}
0
\leq
\langle \gamma^*, C-D^{(f)}\rangle
=
\langle \gamma^*,C\rangle
-
\langle \gamma^*,D^{(f)}\rangle
=
W_1(\mu,\nu)-L(f)
=
0.
\end{equation*}
Since $C-D^{(f)}$ is nonnegative and $\gamma^*(A)\geq0$ for all (measurable) $A\subset\bbR^d\times \bbR^d$, this implies $C-D^{(f)} = 0$ $\gamma^*$-almost everywhere.
Hence $\gamma^*$ is supported on $S_f$. 
Thus, (iv) and (v) hold.

Finally, suppose (iv) holds. Let $\gamma_f\in\Gamma_f$ be supported on $S_f$. Then
\begin{equation*}
L(f)
=
\langle \gamma_f,D^{(f)}\rangle
=
\langle \gamma_f,C\rangle,
\end{equation*}
because $D^{(f)}=C$ on the support of $\gamma_f$. Since $\gamma_f$ is feasible for the
ambient problem,
\begin{equation*}
W_1(\mu,\nu)\leq \langle \gamma_f,C\rangle = L(f).
\end{equation*}
Together with \eqref{eq:sandwich_exact_slicer}, this gives $L(f)=W_1(\mu,\nu)$.
Thus, $\gamma_f\in \Gamma_C^*$, (iv) implies (i),
and the proof is complete.
\end{proof}

\begin{remark}
    As a consequence of the proofs of Proposition \ref{prop:KR_simultaneous_no_uniqueness} and Theorem \ref{thm: characterization}, if $f\in \mathcal{F}_1$ is such that $L(f)=W_1(\mu,\nu)$,  then every ambient optimal plan belongs to $\Gamma_f$, i.e., if $\gamma^*\in \Gamma_C^*$ then $\gamma^*\in \Gamma_f$. However, the converse fails. For example, if $f\equiv c$ for some constant $c$, then $D^{(f)}\equiv 0$, so $\Gamma_f=\Gamma(\mu,\nu)$ and therefore every $\gamma^*\in \Gamma_C^*$ belongs to $\Gamma_f$, but $L(f)=0\not=W_1(\mu,\nu)$ (unless $\mu\equiv \nu$). This phenomenon is also illustrated by the endpoint $\lambda=0$ of the scaling path in Example \ref{example: scaling path}.
    This is the reason why part (v) in Theorem \ref{thm: characterization} requires an extra assumption on the support of the optimal couplings $\gamma^*$. 
\end{remark}

Finally, the next corollary reformulates exactness, namely $L(f)=W_1(\mu,\nu)=U(f)$, as a factorization property of optimal KR potentials. Part \textit{(iii)} of Theorem \ref{thm: characterization} establishes that an optimal ambient KR potential $\phi^*$ factors through the slicer $f$ when the gap closes ($L(f)=U(f)$), i.e.,  $\phi^*=h_f\circ f$. Corollary \ref{coro: characterization} focuses on the role of $h_f$ as an optimal KR potential for the sliced Wasserstein problem. Although the proof is a repetition of the technique used in the proofs of Lemma \ref{lem: lower_bound_attained}, Proposition \ref{prop:KR_simultaneous_no_uniqueness} and Theorem \ref{thm: characterization}, we include it here for completeness. 

\begin{corollary}\label{coro: characterization}
Let $f:\mathbb R^d\to\mathbb R$ be $1$-Lipschitz. 
Then $L(f)=W_1(\mu,\nu)$ if and only if there exists 
an optimal one-dimensional KR potential $h_f:\mathbb R\to\mathbb R$ for $W_1(f_\#\mu,f_\#\nu)$ such that $h_f\circ f$ is an optimal KR potential for $W_1(\mu,\nu)$.
Moreover, if $L(f)=W_1(\mu,\nu)$, then every optimal one-dimensional KR potential $h:\bbR\to\bbR$ for $W_1(f_\#\mu,f_\#\nu)$ has this property. 
In particular, if $f$ itself is an optimal KR potential for $W_1(\mu,\nu)$, then the identity map $h(t)=t$ is an optimal one-dimensional KR potential for $W_1(f_\#\mu,f_\#\nu)$.
\end{corollary}

\begin{proof}
If $L(f)=W_1(\mu,\nu)$ and $h_f$ is any optimal KR potential for $W_1(f_\#\mu,f_\#\nu)$, then
\begin{small}
    \[
\int h_f\circ f\,\dd\mu-\int h_f\circ f\,\dd\nu
= W_1(f_\#\mu,f_\#\nu)=L(f)=W_1(\mu,\nu),
\]
\end{small}
\vspace{-0.15in}

\noindent and, since $h_f\circ f$ is $1$-Lipschitz, it is an optimal ambient KR potential. Conversely, if $h_f$ is optimal for the projected problem and $h_f\circ f$ is ambient optimal, the same identity gives $W_1(\mu,\nu)=W_1(f_\#\mu,f_\#\nu)=L(f)$.
Finally, suppose that $f$ itself is an optimal KR potential for $W_1(\mu,\nu)$. 
Taking $h(t)=t$, which is in particular  $1$-Lipschitz,  we have
\begin{small}
\begin{equation*}
W_1(\mu,\nu)=\int_{\bbR^d}  f \, \dd \mu - \int_{\bbR^d} f \, \dd \nu = \int_{\bbR}  t\, \dd f_{\#}\mu(t) - \int_{\bbR}  t\, \dd f_{\#}\nu(t) \leq W_1(f_\#\mu,f_\#\nu)=L(f)\leq W_1(\mu,\nu),
\end{equation*}
\end{small}
\vspace{-0.15in}

\noindent implying $W_1(f_\#\mu,f_\#\nu) = L(f) = W_1(\mu,\nu)$.
Therefore $h(t)=t$ attains the one-dimensional KR dual for $W_1(f_\#\mu,f_\#\nu)$, as claimed.
\end{proof}

\section{Parameterized Slicer Classes}
\label{sec:lip_nets}

Proposition~\ref{prop:KR_simultaneous_no_uniqueness} identifies an optimal KR
potential $\phi^*\in\mathcal{F}_1$ as the object that closes the sandwich
$L^*=W_1=U^*$. In practice $\phi^*$ is unknown, so one can parameterize
$\mathcal{F}_1$ by a family of neural networks and maximize $L(f)$ (or
minimize $U(f)$) over the parameters. For the resulting $L^*$ to
remain a valid lower bound and $U^*$ a valid upper bound, every
network in the family must be \emph{genuinely} $1$-Lipschitz: the pointwise
inequality $D^{(f)}\le C$ in~\eqref{eq:D_le_C} is what the bounds rest on, and
it holds only when $\Lip(f)\le 1$. We therefore restrict $\mathcal{F}_1$ to
architectures that enforce this constraint by construction rather than as a
soft penalty. Fortunately, this problem has been well studied in the literature. For completeness, we provide a brief overview of recent advances in this direction below.
We consider three candidate parameterizations---a norm-constrained GroupSort
network, random Fourier features with a norm-constrained linear read-out, and
the same feature map with the frequencies trained as well---and compare them on
an exactly $1$-Lipschitz target before selecting the one used in the rest of the
paper.

\paragraph{Composition and the affine building block}
By the composition property, a network $f=g_L\circ\sigma\circ
g_{L-1}\circ\cdots\circ\sigma\circ g_1$ is $1$-Lipschitz whenever every affine
map $g_\ell(z)=W_\ell z+b_\ell$ and every activation $\sigma$ is
$1$-Lipschitz~\cite{anil2019sorting}. For the affine maps with respect to the
Euclidean norm, this means controlling the spectral norm (largest singular
value) $\sigma_{\max}(W_\ell)\le 1$. Following
Miyato et al.~\cite{miyato2018spectral}, the spectral norm is estimated by
power iteration: maintaining left/right singular-vector estimates
$\tilde u,\tilde v$ updated by
\begin{equation*}
\tilde v \leftarrow \frac{W^\top \tilde u}{\|W^\top \tilde u\|_2},
\qquad
\tilde u \leftarrow \frac{W \tilde v}{\|W \tilde v\|_2},
\qquad
\hat\sigma(W) = \tilde u^\top W \tilde v,
\vspace{-0.075in}
\end{equation*}
and dividing $W$ by $\hat\sigma(W)$ is intended to yield a layer with spectral norm at most
one. 
A single iteration per step suffices in practice when $\tilde u$ is
warm-started from the previous step~\cite{miyato2018spectral}. We note one
caveat relevant to using the bounds as \emph{certificates} rather than
training surrogates: power iteration converges to $\sigma_{\max}$ from below, so
$\hat\sigma(W)\le\sigma_{\max}(W)$, and dividing by an underestimate can leave
the realized operator norm slightly above one. When a hard guarantee on
$D^{(f)}\le C$ is needed, the spectral norm should be evaluated exactly (e.g.\
via the SVD) at the point of certification.

\paragraph{Gradient-norm preservation and the choice of activation}
A subtlety specific to $1$-Lipschitz settings is that spectral-norm
constraints alone yield poor expressivity if the activation contracts the
gradient. Anil et al.~\cite{anil2019sorting} show that a norm-constrained
network with monotone elementwise activations (ReLU, sigmoid, tanh) whose
gradient has unit norm almost everywhere is necessarily \emph{linear}: a
$2$-norm-constrained ReLU network cannot represent even the absolute-value
function, because gradient-norm preservation forces the activations to stay in
their linear region. They resolve this with \emph{GroupSort}, which partitions
the pre-activations into groups and sorts each group; with group size two this
is the \emph{MaxMin} activation $(a,b)\mapsto(\max(a,b),\min(a,b))$. GroupSort
is $1$-Lipschitz and gradient-norm preserving, and they prove that
norm-constrained GroupSort networks are \emph{universal} approximators of
$1$-Lipschitz functions~\cite{anil2019sorting}. Tanielian et
al.~\cite{tanielian2021approximating} sharpen this into a quantitative
statement: norm-constrained GroupSort networks represent any Lipschitz
continuous piecewise-linear function exactly, and they give explicit upper
bounds on the depth and width required to approximate a $1$-Lipschitz function
to a prescribed accuracy. This universality is precisely what we want for
$\mathcal{F}_1$: it is the property under which the parameterized class can, in
principle, contain an optimal KR potential $\phi^*$ and thereby attain the
tight identities of Proposition~\ref{prop:KR_simultaneous_no_uniqueness}, while the approximation rates indicate the network size needed to come close in
practice. A linear slicer class cannot do so in general, consistent with the
failure of exact linear max-sliced/$W_1$ equivalence for $d\ge
2$~\cite{bayraktar2024errata}.
Finally, we note that the universality results of \cite{anil2019sorting,tanielian2021approximating} are established under the specific mixed operator-norm constraints considered therein, while in our experiments we use spectral normalization of the affine layers with respect to the Euclidean norm. Thus, these results motivate our choice of GroupSort but do not provide a formal universality guarantee for the particular architecture implemented here.

\paragraph{Fourier-feature slicers}
Constraining a deep network is not the only route to $\mathcal{F}_1$. A
cheaper alternative fixes a nonlinear feature map and constrains only a linear
read-out. Following the random-feature approximation of Rahimi and
Recht~\cite{rahimi2007random}, set
\begin{equation*}
z_{\Omega,b}(x)
\coloneqq
\sqrt{\tfrac{2}{m}}
\bigl(
\cos(\omega_1^\top x+b_1),\ldots,
\cos(\omega_m^\top x+b_m)
\bigr)^\top,
\qquad
f_a(x)\coloneqq a^\top z_{\Omega,b}(x),
\end{equation*}
and let $\Omega$ denote the $m\times d$ matrix with rows $\omega_k^\top$.
Differentiating,
$\nabla f_a(x)
=
-\sqrt{2/m}\,\Omega^\top\bigl(a\odot\sin(\Omega x+b)\bigr)$,
where $\odot$ is the entrywise product, and since $|\sin(\theta)|\leq1$ entrywise,
\begin{equation}\label{eq:rff_lipschitz}
\operatorname{Lip}(f_a)
\leq
\kappa_\Omega\|a\|_2,
\qquad
\kappa_\Omega
\coloneqq
\sqrt{\tfrac{2}{m}}\,\|\Omega\|_{\mathrm{op}},
\end{equation}
with $\|\cdot\|_{\mathrm{op}}$ the largest singular value. Imposing
$\|a\|_2\leq\kappa_\Omega^{-1}$ therefore makes every admissible slicer
$1$-Lipschitz. When $(\Omega,b)$ are drawn once and \emph{frozen},
$\kappa_\Omega$ is a constant of the draw and only the $m$ coefficients are
trained, so fitting the slicer reduces to a norm-constrained linear problem in
a fixed nonlinear feature space. This is the random Fourier feature (RFF)
class used in the appendix experiment of Figure~\ref{fig:RFF} (see Appendix \ref{sec:app:exp-RFF}). Throughout we
draw $\omega_k\sim\mathcal N(0,\ell^{-2}I_d)$ and $b_k\sim\mathrm{Unif}[0,2\pi)$,
the standard random-feature approximation to a Gaussian kernel of length-scale
$\ell$, and fix $\ell$ by the median heuristic, i.e., the median pairwise
distance of the data.

\paragraph{Learnable Fourier features}
Between the two constructions sits a third: keep the feature map but train it,
initializing $(\Omega,b)$ at random as above and optimizing them jointly with
$a$. The bound \eqref{eq:rff_lipschitz} still holds, but $\kappa_\Omega$ is no
longer a constant, so it must be recomputed at every evaluation, and the
caveat above about power iteration applies to $\|\Omega\|_{\mathrm{op}}$ exactly
as it does to the network's layer norms. The gain over the frozen class is not
merely a matter of additional parameters. Because the constraint couples its
two factors, the model can \emph{shrink} $\|\Omega\|_{\mathrm{op}}$ in order to
enlarge the admissible ball for $a$, trading a lower-frequency basis for a
larger amplitude budget; the frozen class has no such freedom.

\paragraph{Comparing the three classes}
To choose among them we fit all three to a common target that is exactly
$1$-Lipschitz: the signed distance function (SDF) of a plus-shaped region,
constructed as the pointwise minimum of two box SDFs. Since each box SDF is
$1$-Lipschitz under the $\ell_2$ norm and a pointwise minimum of $1$-Lipschitz
functions preserves this property, the target satisfies $\|\nabla f\|_2=1$ almost
everywhere, with creases along the medial axis where the two branches meet. It
is a demanding but fair probe: an admissible slicer may match it exactly, and
none may exceed it. All three classes are trained under identical conditions:
the same $50{,}000$-point sample on $[-2.5,2.5]^2$, the same batch size, Adam
with a cosine-annealed step size, $10^4$ iterations, and the same seed. Results
are collected in Table~\ref{tab:slicer-comparison} and
Figure~\ref{fig:slicer-comparison}.

\emph{Choosing the capacity of each class.} Matching trainable-parameter counts
across these families is neither possible nor meaningful: a frozen RFF slicer
has exactly $m$ trainable coefficients however wide its feature map, so
equalizing against a network of a few million parameters would demand an
infeasible $m$. Instead we give each class enough capacity that more stops
helping, verified by a sweep over each family's capacity knob---including the
length-scale that the Fourier classes carry and the network has no analogue
of, so that neither is handicapped by a heuristic left untouched. We report
the network with four hidden layers of width
$1024$, and both Fourier classes at $m=4096$ with the median-heuristic
bandwidth---the frozen class's own optimum, and a setting at which the
learnable class is indistinguishable from its best---so that the only
difference between the two Fourier rows is whether the frequencies are trained.

\emph{What the comparison shows.} Two quantities matter. The first is fit
quality, reported as $R^2$. The second is the largest gradient norm
$\max\|\nabla f\|_2$ the trained slicer attains, which on an exactly
$1$-Lipschitz target should approach $1$: a slicer whose gradient norm stays
well below $1$ is not using the budget it was given, so its Lipschitz
certificate is loose and it forfeits expressivity it has already paid for. The
frozen RFF class fails on both counts, and the first failure explains the
second. It saturates its certificate, in the sense that
$\kappa_\Omega\|a\|_2=1$ at the optimum, while realizing a maximum gradient
norm of only $0.530$; the bound \eqref{eq:rff_lipschitz} is attained only if
$a\odot\sin(\Omega x+b)$ aligns with the top singular direction of $\Omega$ at
some $x$, which a frozen random draw does not arrange, so roughly half the
Lipschitz budget is unusable however many features are added. The recovered
function is correspondingly over-smoothed, with no zero level set at all inside
the domain (Figure~\ref{fig:slicer-comparison}, third column).

Training the frequencies removes exactly this deficit, and by the mechanism
anticipated above: $\kappa_\Omega$ falls from $0.566$ to $0.388$ as the model
shrinks $\|\Omega\|_{\mathrm{op}}$ to widen the ball for $a$, and the realized
gradient norm rises to $0.998$, so the learnable class now uses essentially all
of its budget. Its residual error is therefore no longer a matter of the
constraint but of the basis: a finite sum of smooth cosines cannot produce a
crease, so its zero level set is a circle where the target's is a plus, and
$R^2$ stops at $0.967$. The GroupSort network is the only class that both uses
its budget and represents the shape: $R^2=0.9999$, a maximum pointwise error of
$0.025$, and a gradient norm of $0.995$ that is visually uniform over the domain
(Figure~\ref{fig:slicer-comparison}, second row). Its remaining discrepancies
sit exactly where the target is non-differentiable, along the diagonal creases
and the central crossing, where the true subgradient is set-valued and any
continuous network must transition smoothly rather than switch abruptly. This
is the expected and benign failure mode: the network cannot reproduce a sharp
gradient discontinuity, so it undershoots the unit norm in a thin neighborhood
of the medial axis while never violating the bound.

We therefore carry out all subsequent experiments with the $1$-Lipschitz
network. We stress that this is a practical choice and not a theoretical one.
Theorem~\ref{thm: characterization} applies to any $1$-Lipschitz $f$, and
Remark~\ref{rem:slicer_class} already records that $L^*$ and $U^*$ remain valid
bounds for any restricted $\mathcal{F}_1$. But the exact identities of
Proposition~\ref{prop:KR_simultaneous_no_uniqueness} are attained only when
$\mathcal{F}_1$ contains an optimal KR potential, so among candidate
parameterizations the one that approximates $1$-Lipschitz functions best is the
one most likely to realize them; and of the three, only the norm-constrained
GroupSort network carries a universal approximation
guarantee~\cite{anil2019sorting,tanielian2021approximating}.

\begin{table}[t!]
  \centering
  \begin{tabular}{lrccccc}
    \hline
    Slicer class & Params & $R^2$ & RMSE & $\max|e|$ & $\max\|\nabla f\|_2$ & ratio \\
    \hline
    $1$-Lipschitz network  & 3{,}152{,}897 & 0.9999 & 0.0070 & 0.025 & 0.995 & 1.00 \\
    RFF (frozen)           &       4{,}097 & 0.7454 & 0.3147 & 0.957 & 0.530 & 0.53 \\
    Fourier (learnable)    &      16{,}385 & 0.9672 & 0.1129 & 0.337 & 0.998 & 1.00 \\
    \hline
  \end{tabular}
  \caption{The three slicer classes fit to the plus-shaped SDF under identical
    training conditions, each at the capacity at which its own family
    saturates. ``Params'' counts trainable parameters,
    $\max|e|$ is the largest pointwise error over the domain, and ``ratio'' is
    $\max\|\nabla f\|_2$ divided by the class's certified Lipschitz bound, which
    equals $1$ for all three: it measures how much of the available Lipschitz
    budget the trained slicer actually uses. The network and the learnable
    Fourier class use essentially all of theirs; the frozen RFF class saturates
    its certificate while realizing barely half of it, which is why it
    under-fits regardless of how many features it is given.}
  \label{tab:slicer-comparison}
\end{table}

\begin{figure}[t!]
  \centering
  \includegraphics[width=\linewidth]{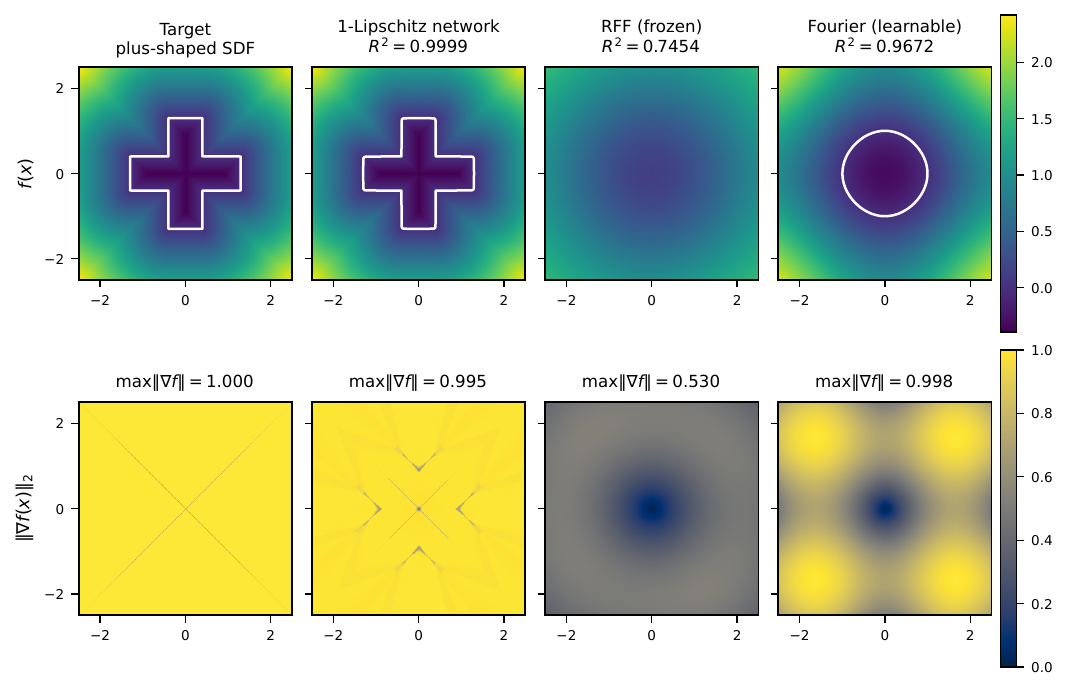}
  \vspace{-0.2in}
  \caption{The three candidate slicer classes on an exactly $1$-Lipschitz
    target. \textbf{Top:} the plus-shaped SDF and the function recovered by
    each class, on a common colour scale, with the zero level set in white.
    \textbf{Bottom:} the corresponding gradient-norm fields $\|\nabla f\|_2$,
    with the colour scale capped at $1$, so that any panel which is not
    saturated lies strictly inside the constraint and the colour itself reports
    the unused Lipschitz budget. The target has $\|\nabla f\|_2=1$ almost
    everywhere. Frozen random features are strongly over-smoothed and reach
    only $\max\|\nabla f\|_2=0.530$, leaving half the budget unused; training the
    frequencies raises this to $0.998$, so the learnable class spends its whole
    budget, but a finite sum of smooth cosines still cannot form a crease and
    its zero level set is a circle where the target's is a plus. The GroupSort
    network recovers the shape and attains a visually uniform unit gradient
    norm, undershooting only in a thin neighborhood of the medial-axis creases,
    where the target is non-differentiable and any continuous network must.}
  \label{fig:slicer-comparison}
\end{figure}

\section{Experiments}
\label{sec:experiments}

Although the preceding results hold for arbitrary measures in $\cP_1(\bbR^d)$, the numerical
experiments below use finite empirical measures, so that $W_1$, $L(f)$ and $U(f)$ can be evaluated exactly by finite-dimensional optimization. This is a
specialization of the evaluation procedure, not of the theory.

We now illustrate Proposition~\ref{prop:KR_simultaneous_no_uniqueness} empirically. The
theory makes a sharp prediction: if the slicer class $\mathcal{F}_1$ contains an
optimal KR potential $\phi^*$, then a single learned slicer
$f_\theta \approx \phi^*$ should drive the max-sliced lower objective $L(f_\theta)$
\emph{up} to $W_1(\mu,\nu)$ and the min-sliced upper objective
$U(f_\theta)$ \emph{down} to the same value, closing the sandwich
$L^* \le W_1 \le U^*$ from both sides at once. We test this on four
two-dimensional configurations of increasing difficulty.

\paragraph{Setup} In each experiment $\mu=\sum_{i=1}^M \alpha_i\delta_{x_i}$ and $\nu=\sum_{i=1}^N \beta_j \delta_{y_j}$ (where $\alpha_i, \beta_j\geq 0$, $\sum_{i=1}^M \alpha_i = 1 = \sum_{j=1}^N \beta_j$) are empirical measures on
$\mathbb{R}^2$. We parameterize $\mathcal{F}_1$ by the exact $1$-Lipschitz
GroupSort architecture of Section~\ref{sec:lip_nets}, here at three hidden
layers of width $4096$ with group size $32$ and $32$ power iterations per step:
the slicer is fit to a handful of points rather than to a field, so capacity is
cheap and is not the quantity under study. We train $f_\theta$ by maximizing
the KR dual objective $J_{\mathrm{KR}}(f_\theta) = \sum_i \alpha_i f_\theta(x_i) - \sum_j \beta_j f_\theta(y_j)$
with Adam for $2048$ iterations, at a step size warmed up geometrically over the
first $256$ and cosine-annealed thereafter. Throughout
training we monitor four scalars: (i) the KR dual objective $J_{\mathrm{KR}}$ that is being
optimized; (ii) the one-dimensional sliced distance $W_1(f_{\theta\#}\mu,
f_{\theta\#}\nu)$, computed by sorted matching and \emph{not} backpropagated;
(iii) the tie-broken upper evaluation $U(f_\theta)$; and (iv) the true
ambient distance $W_1(\mu,\nu)$, computed once by linear programming with
\texttt{ot.emd2} and drawn as a fixed reference line.

\paragraph{Computing the tie-broken upper objective} Evaluating
$U(f_\theta)$
is the lexicographic program of equation~\eqref{eq:def_U_tilde}: first minimize
the slice cost $D^{(f_\theta)}$, then minimize the ambient cost $C$ among
slice-optimal couplings. We realize this with a single linear program on the
perturbed cost $D^{(f_\theta)} + \epsilon\, C$ with $\epsilon = 10^{-6}$. 
We show in Appendix~\ref{sec:app:costpert} that perturbing the cost in this way is a consistent approximation.

\paragraph{Results}
Figure~\ref{fig:kr-experiments} reports the four configurations, one per row.
Each row shows (a) the learned map $f_\theta$ over the plane with the two point
clouds overlaid, (b) the one-dimensional pushforwards $f_{\theta\#}\mu$ and
$f_{\theta\#}\nu$, and (c) the four monitored quantities against the iteration
count, with the interval $[L(f_\theta), U(f_\theta)]$ they certify around $W_1$
shaded.

The qualitative behavior predicted by Proposition~\ref{prop:KR_simultaneous_no_uniqueness} is
visible across all four. The tie-broken upper curve $U(f_\theta)$
(green) descends onto the LP reference for $W_1$ (black dashed) and stays there,
the descent reflecting an upper bound that starts loose at initialization and is
pulled to tightness as $f_\theta \to \phi^*$. That descent is a staircase rather
than a curve, and a brief one. It is a staircase because $U(f_\theta)$ is the
ambient cost of whichever coupling the tie-broken program returns, so it is
piecewise constant in $\theta$ and moves only when the slice-optimal coupling
itself changes. It is brief because at a constant step size it is over within
roughly ten iterations; the warmup above spreads that same descent across about
$150$ without changing the values at convergence, and panel~(c) uses a
logarithmic iteration axis so that it is legible at all. In the
separable, overlapping, and radial
configurations (rows~1, 2, and 4 of Fig.~\ref{fig:kr-experiments}) all four
quantities coincide on $W_1$ at convergence, realizing the exact collapse
$L(\phi^*) = W_1 = U(\phi^*)$.

The radial case (row~4 of Fig.~\ref{fig:kr-experiments}) is the most direct
confirmation, because here the optimal KR potential is known in closed form: for
two concentric circles the radial function $\phi^*(x) = \|x\|_2$ is optimal, and the
learned $f_\theta$ recovers it: panel (a) shows concentric level sets, and panel
(b) collapses each pushforward to a near-atom at the corresponding radius, so the
one-dimensional problem reduces to moving mass between two points and
$U$, $L$, and $W_1$ agree at the difference of radii.

\paragraph{When the slicer class is not expressive enough} The interleaved-spiral
configuration (row~3 of Fig.~\ref{fig:kr-experiments}) is the exception that
confirms the scope of the theory. The two lower curves (the KR dual and the sliced distance) plateau
just below the LP reference and do not reach it within the training budget,
while $U(f_\theta)$ still descends onto $W_1$. This is precisely the
regime described in our scope discussion (Remark~\ref{rem:slicer_class}): an
exactly $1$-Lipschitz network of fixed capacity need not contain an optimal KR
potential for a strongly entangled geometry, so $L^*$ remains a valid but not
tight lower bound. The gap is an expressivity limitation of the finite network,
not a failure of the identity; it narrows as the model is enlarged, consistent
with the universal-approximation guarantees for \textsc{GroupSort} networks
invoked in Section~\ref{sec:lip_nets}. The asymmetry is itself
informative: the upper identity $U(\phi^*) = W_1$ is attained through
the tie-broken lift and is robust to the slicer being only approximately optimal,
whereas the lower identity $L(\phi^*) = W_1$ inherits the full approximation
burden of realizing $\phi^*$.

\begin{figure}[t!]
  \centering
  \includegraphics[width=\linewidth]{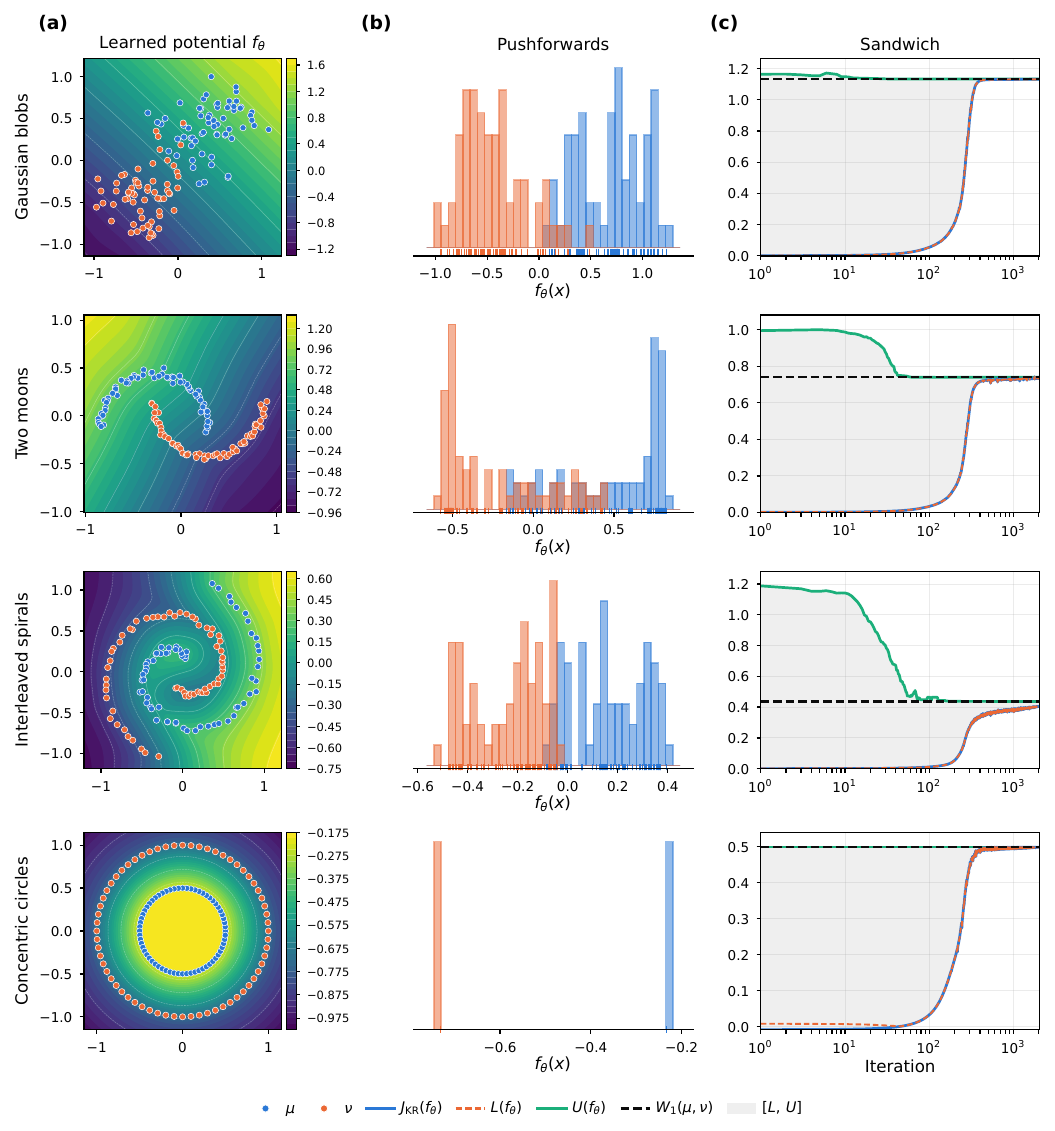}
 \vspace{-0.2in}\caption{Empirical illustration of
    Proposition~\ref{prop:KR_simultaneous_no_uniqueness} on four
    configurations, one per row: (a) the learned $1$-Lipschitz potential
    $f_\theta$ with $\mu$ and $\nu$ overlaid; (b) the induced pushforwards;
    (c) the monitored quantities, with the certified interval
    $[L(f_\theta), U(f_\theta)]$ shaded and a logarithmic iteration axis.
    Section~\ref{sec:experiments} discusses the rows.}
 \label{fig:kr-experiments}
\end{figure}

\paragraph{Why the KR dual and the sliced distance coincide} A feature visible
in panel (c) in Fig. \ref{fig:kr-experiments} of every figure is that the optimized KR dual objective (blue) and
the monitored one-dimensional sliced distance (orange) track each other almost
exactly throughout training, despite the latter never being back-propagated. This
is not a coincidence but the numerical signature of the chain in
equation~\eqref{eq:every_coupling_slice_lower}. Once $f_\theta$ approaches an optimal
potential it \emph{sign-separates} the pushforwards: panel (b) in Fig. \ref{fig:kr-experiments} shows
$f_{\theta\#}\mu$ supported to the left of $f_{\theta\#}\nu$ with little overlap.
On such configurations every pair matched by the one-dimensional optimal coupling
is displaced in the same direction, so $|f_\theta(x_i) - f_\theta(y_j)|$ unfolds
to $f_\theta(x_i) - f_\theta(y_j)$ along the matching, the sorted-matching cost
telescopes, and the sliced distance equals
$\sum_i \alpha_i f_\theta(x_i) - \sum_j \beta_j f_\theta(y_j)$, the KR dual value.
In other words the inequality in~\eqref{eq:every_coupling_slice_lower} closes to an equality
exactly when the slicer sign-separates the slices, which an optimal $\phi^*$ does;
the coincident curves are the lower identity manifesting in finite samples. We
emphasize that this equality is special to sign-separating slicers and is not an
algebraic identity for arbitrary $f$.

\subsection{Exact slicers need not be optimal potentials}
\label{sec:factorization}

The characterization of Theorem~\ref{thm: characterization} asks for a one-dimensional
$1$-Lipschitz $h$ with $h \circ f$ an optimal ambient KR potential, which is
strictly weaker than asking $f$ to be one itself. The distinction is easy to
miss, because on the configurations of Section~\ref{sec:experiments} the
learned slicer happens to be optimal on its own. It is not a technicality: the
restricted objective $J_{\mathrm{KR}}(f) = \mathbb{E}_\mu f - \mathbb{E}_\nu f$
can be blind to a slicer that is exact.

\paragraph{An exact slicer with zero KR value} Take
$\mu = \tfrac12(\delta_{(-2,0)} + \delta_{(2,0)})$ and
$\nu = \tfrac12(\delta_{(-1,0)} + \delta_{(1,0)})$ with the linear slicer
$f(x) = x_1$. Then $W_1(\mu,\nu) = L(f) = U(f) = 1$, so the sandwich is closed
and $f$ is exact, while $J_{\mathrm{KR}}(f) = 0$ exactly (the two atoms of each
measure are symmetric about the origin and cancel in the mean). On the line the
optimal potential is available in closed form (integrating by parts, the
supremum over $\mathrm{Lip}(h) \le 1$ is attained by
$h'(t) = \operatorname{sign}(F_{f_\#\nu}(t) - F_{f_\#\mu}(t))$, so $h_f$ is
piecewise linear with slopes in $\{-1,0,1\}$ and knots at the support) and
recovering it gives $h_f(t) = |t|$ on the support, with
$J_{\mathrm{KR}}(h_f \circ f) = W_1$. Figure~\ref{fig:factorization} shows the
three objects. Numerically $|J_{\mathrm{KR}}(f)|/W_1$, $|L(f)-W_1|/W_1$ and
$|U(f)-W_1|/W_1$ all vanish to double precision, and
$\mathrm{Lip}(h_f) = 1$.

\begin{figure}[t!]
  \centering
  \includegraphics[width=\linewidth]{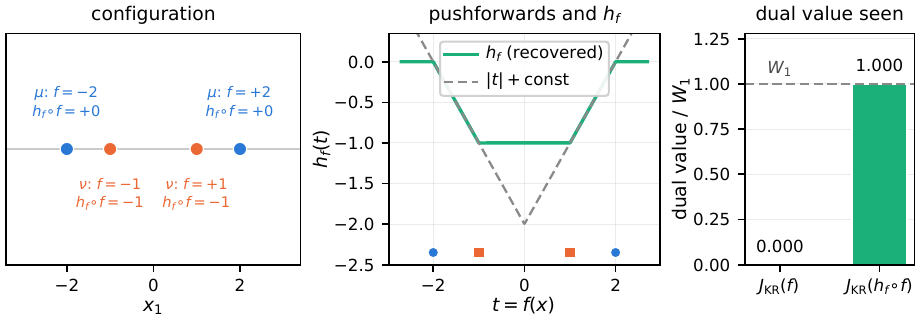}
  \vspace{-0.2in}
  \caption{An exact slicer that is not an optimal potential. Left: the four
    atoms, each labelled with $f$ and with $h_f \circ f$; $f$ takes opposite
    signs on the two atoms of each measure and cancels in the mean, while
    $h_f \circ f$ takes the same value on both and does not. Middle: the
    pushforwards and the recovered $h_f$, which agrees with $|t|$ on the
    support and is flat between the inner atoms, where every slope in $[-1,1]$
    is equally optimal. Right: the dual value each objective sees.}
  \label{fig:factorization}
\end{figure}

\paragraph{Does the distinction appear in training?} Maximizing $L(f)$ directly
rather than $J_{\mathrm{KR}}(f)$ is what the characterization calls for, and
the two can be compared with everything else held fixed---architecture,
initialization, optimizer, schedule, budget and data. Table~\ref{tab:objective-comparison}
reports medians and interquartile ranges over ten seeds. Because $L$ is
invariant to $f \mapsto -f$ while $J_{\mathrm{KR}}$ is not, seeds trained on
$L$ split between the two orientations and we report $|J_{\mathrm{KR}}|$.

On the separable, overlapping and radial configurations the two objectives are
indistinguishable and $|J_{\mathrm{KR}}|/L = 1.000$: the learned slicer is
already an optimal potential and there is nothing for $h_f$ to do. On the
interleaved spirals they separate. Maximizing $L$ returns slicers with
$L/W_1 = 0.813$ but $|J_{\mathrm{KR}}|/L = 0.086$, in ten seeds out of ten---%
informative scalar representations that the restricted objective scores at
essentially zero, and for which $J_{\mathrm{KR}}(h_f \circ f) = L(f)$ recovers
the whole value. The phenomenon of the four-point example is therefore not an
artifact of that construction. It comes at a price in optimization: the sliced
objective has a rougher landscape, and its sandwich gap is both wider
($0.264$ against $0.074$) and far more variable across seeds.

\begin{table}[t!]
  \centering
  \begin{tabular}{llcccc}
      \hline
      Configuration & Trained on & $|J_{\mathrm{KR}}|/W_1$ & $L/W_1$ & $|J_{\mathrm{KR}}|/L$ & $(U-L)/W_1$ \\
      \hline
      Gaussian blobs         & $J_{\mathrm{KR}}$    & 0.998 & 0.998 & 1.000 & 0.003 {\scriptsize [0.002, 0.003]} \\
                             & $L$                  & 0.998 & 0.998 & 1.000 & 0.003 {\scriptsize [0.003, 0.003]} \\
      \hline
      Two moons              & $J_{\mathrm{KR}}$    & 0.992 & 0.992 & 1.000 & 0.008 {\scriptsize [0.008, 0.009]} \\
                             & $L$                  & 0.991 & 0.991 & 1.000 & 0.009 {\scriptsize [0.009, 0.009]} \\
      \hline
      Interleaved spirals    & $J_{\mathrm{KR}}$    & 0.926 & 0.926 & 1.000 & 0.074 {\scriptsize [0.073, 0.074]} \\
                             & $L$                  & 0.070 & 0.813 & 0.086 & 0.264 {\scriptsize [0.259, 0.443]} \\
      \hline
      Concentric circles     & $J_{\mathrm{KR}}$    & 0.997 & 0.997 & 1.000 & 0.003 {\scriptsize [0.003, 0.003]} \\
                             & $L$                  & 0.997 & 0.997 & 1.000 & 0.003 {\scriptsize [0.003, 0.003]} \\
      \hline
    \end{tabular}
  \caption{Maximizing the restricted objective $J_{\mathrm{KR}}(f)$ against maximizing $L(f)$ directly, on the four configurations of 
  Figure \ref{fig:kr-experiments}.
  Medians over ten seeds, with the interquartile range of the sandwich gap in brackets; all quantities are evaluated on the certified path and normalized by $W_1(\mu,\nu)$. $|J_{\mathrm{KR}}|/L = 1$ says the learned slicer is itself an optimal potential. Only the entangled configuration separates the two objectives.}
  \label{tab:objective-comparison}
\end{table}

\paragraph{What the remaining slack is made of} For any ambient-optimal
$\gamma^\star$ the shortfall splits exactly,
\begin{equation}
  W_1 - L(f)
  = \underbrace{\langle \gamma^\star,\, C - D^{(f)} \rangle}_{\text{Lipschitz saturation}}
  + \underbrace{\bigl( \langle \gamma^\star,\, D^{(f)} \rangle - L(f) \bigr)}_{\text{slice optimality}},
  \label{eq:calibration-split}
\end{equation}
and both terms are nonnegative: the first because $D^{(f)} \le C$ pointwise for
$1$-Lipschitz $f$, the second because $L(f)$ minimizes
$\langle \gamma, D^{(f)} \rangle$ over couplings while $\gamma^\star$ is merely
feasible. Reading them together is what makes the statement operational: the
sandwich closes exactly when the ambient-optimal plans become slice-optimal
\emph{and} the Lipschitz inequality is saturated along their transported pairs.
Figure~\ref{fig:calibration} tracks the split. Trained on $J_{\mathrm{KR}}$ the
slice-optimality term vanishes and the entire residual is saturation; trained
on $L$ both stay open, which is the same non-potential solution the table
records. On these configurations the ambient optimum is a unique vertex, so the
split is well defined; the degeneracy that makes tie-breaking necessary for
$U(f)$ lives in the slice-optimal face $\Gamma_f$, not in the ambient one.

\begin{figure}[t!]
  \centering
 \includegraphics[width=0.9\linewidth]{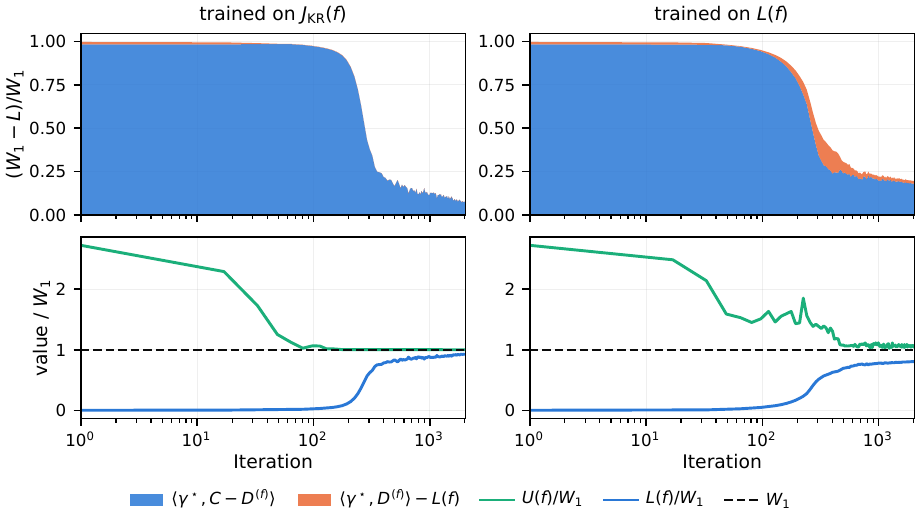}
\vspace{-0.25in}
  \caption{The shortfall $W_1 - L(f)$ on the interleaved spirals, decomposed by
    \eqref{eq:calibration-split} as a stacked area, with the sandwich itself
    below. The two terms sum to the total at every checkpoint to double
    precision.}
  \label{fig:calibration}
\end{figure}

\section{Conclusion and Discussion}

By working with \(1\)-Lipschitz
scalar slicers \(f\), we have placed max-sliced lower bounds  
and lifted min-sliced upper
bounds for \(W_1\) in a common framework. That is,
for an arbitrary pair of measures $\mu,\nu\in\mathcal{P}_1(\mathbb R^d)$, 
on the one hand, we maximize the lower bounds  
$L(f)=W_1(f_\#\mu,f_\#\nu)\leq W_1(\mu,\nu)$
over a slicer family and, on the other hand, we optimize the upper bounds $W_1(\mu,\nu)\leq \min_{\gamma\in \Gamma_f}\langle\gamma,C\rangle=U(f)$ defined lexicographically by minimizing the ambient
cost over the set of slice-optimal couplings. 
Our main characterization shows that the sandwich closes precisely
when an optimal ambient KR potential $\phi^*$ factors through the slicer as
\(\phi^*=h_f\circ f\), with \(h_f:\mathbb R\to\mathbb R\) \(1\)-Lipschitz. In
particular, every optimal KR potential is itself an exact generalized
slicer and simultaneously attains both sides of the sandwich.

The analysis also clarifies the role of the two bounds. For instance, the equality \(U(f)=W_1(\mu,\nu)\) alone is weak, as it may already
hold for constants or highly compressed slicers. The informative
condition is instead the closure \(L(f)=U(f)\), or equivalently the
factorization property. Our experiments illustrate the benefit
of expressive \(1\)-Lipschitz slicer classes, including the possible
necessity of nonlinear slicing, and the distinction between directly learning an optimal KR potential and learning a scalar slicer \(f\) through which an optimal KR potential factors as \(h_f\circ f\). 

Natural
next steps include quantitative stability results for approximate
factorizations and scalable methods for closing the sandwich within
restricted slicer classes.





\appendix

\section{Perturbation of Costs}\label{sec:app:costpert}

In this section, we first justify replacing the bilevel optimization problem \eqref{eq:def_U_tilde}
with
\begin{equation}\label{eq: perturbed}
    \min_{\gamma\in\Gamma(\mu,\nu)} \langle \gamma, D^{(f)}_\eps\rangle, \qquad \text{where} \qquad D_\eps^{(f)}(x,y) := |f(x)-f(y)| + \eps \|x-y\|_2,
\end{equation} 
as it selects the lexicographic solution in the limit
\(\epsilon\downarrow0\).
\begin{theorem}
Let $\mu,\nu\in\cP_1(\bbR^d)$ and let $f:\bbR^d\to \bbR$ be \(1\)-Lipschitz.
Then
\[ \lim_{\eps\to 0} \langle \gamma_\eps^*, C\rangle = \min_{\gamma\in \displaystyle \underset{\eta\in\Gamma(\mu,\nu)}{\argmin}\langle \eta,D^{(f)}\rangle}\langle \gamma,C\rangle, \]
where $\gamma^*_\eps\in\Gamma(\mu,\nu)$ is any minimiser of $\gamma\mapsto \langle \gamma,D_\eps^{(f)}\rangle$ over $\gamma\in\Gamma(\mu,\nu)$.
Moreover, $\{\gamma_\eps^*\}_{\eps>0}$ is precompact in the weak$^*$ topology and any convergent subsequence, say $\gamma_\eps^*\weakstarto \gamma^*$, satisfies $\gamma^*\in\Gamma(\mu,\nu)$, $\gamma^*\in\argmin_{\eta\in\Gamma(\mu,\nu)} \langle \eta, D^{(f)}\rangle$ and $\langle\gamma^*,C\rangle \leq \langle \bar{\gamma}, C\rangle$ for all $\bar{\gamma}\in\argmin_{\eta\in\Gamma(\mu,\nu)} \langle \eta,D^{(f)}\rangle$, i.e.,
\[ \langle \gamma^*,C\rangle = \min_{\gamma\in \displaystyle \underset{\eta\in\Gamma(\mu,\nu)}{\argmin}\langle \eta,D^{(f)}\rangle}\langle \gamma,C\rangle. \]
\end{theorem}

\begin{proof}
For convenience we introduce the following notation
\begin{align*}
\cE_\eps(\gamma) &  = \langle \gamma, D_\eps^{(f)}\rangle = \int_{\bbR^d\times \bbR^d} |f(x)-f(y)| + \eps \|x-y\|_2 \, \dd \gamma(x,y) \\
\cE_0(\gamma) & = \langle \gamma, D^{(f)}\rangle = \int_{\bbR^d\times \bbR^d} |f(x)-f(y)| \, \dd \gamma(x,y).
\end{align*}
\emph{Compactness:} 
Let $\gamma_\eps^*\in \Gamma(\mu,\nu)$ be a minimiser of $\cE_\eps$.
Then since $\{\gamma_\eps^*\}_{\eps>0}$ is tight\footnote{Let $\delta>0$ and choose a compact set $K_\delta$ such that $\mu(K_\delta^c)\leq \frac{\delta}{2}$, $\nu(K_\delta^c)\leq \frac{\delta}{2}$, then $\gamma_\eps^*((K_\delta\times K_\delta)^c) \leq \gamma^*_\eps(K_\delta^c \times \bbR^d) + \gamma_\eps^*(\bbR^d\times K_\delta^c) = \mu(K_\delta^c) + \nu(K_\delta^c) \leq \delta$, so $\{\gamma_\eps^*\}_{\eps>0}$ is tight.} then by Prokhorov's theorem $\{\gamma_\eps^*\}_{\eps>0}$ is sequentially compact in the weak$^*$ convergence.
Let $\gamma_\eps^*\weakstarto \gamma^*$ be an arbitrary converging subsequence (which we relabel).
Marginals are conserved with weak$^*$ convergence, so $\gamma^*\in\Gamma(\mu,\nu)$.

\noindent \emph{Liminf:}
Trivially,
\[ \cE_\eps(\gamma_\eps^*) \geq \int_{\bbR^d\times \bbR^d} |f(x)-f(y)| \, \dd \gamma_\eps^*(x,y). \]
Since $f$ is continuous,
$(x,y)\mapsto |f(x)-f(y)|$ is lower semicontinuous and bounded
from below; hence, by the Portmanteau theorem,
\[ \liminf_{\eps\to 0} \cE_\eps(\gamma_\eps^*) \geq \liminf_{\eps\to 0} \int_{\bbR^d\times \bbR^d} |f(x)-f(y)| \, \dd \gamma^*_\eps(x,y) \geq \int_{\bbR^d\times \bbR^d} |f(x)-f(y) |\, \dd \gamma^*(x,y) = \cE_0(\gamma^*). \]
\emph{Limsup:}
Let $\gamma\in\Gamma(\mu,\nu)$ be arbitrary.
Note that
\begin{small}
 \[ \int_{\bbR^d\times\bbR^d} \|x-y\|_2 \, \dd \gamma(x,y) \leq \int_{\bbR^d\times \bbR^d} \|x\|_2 + \|y\|_2 \, \dd \gamma(x,y) \leq \int_{\bbR^d} \|x\|_2 \, \dd \mu(x) + \int_{\bbR^d} \|y\|_2 \, \dd \nu(y) =:M < +\infty. \]   
\end{small}

\noindent We have,
\[ \cE_\eps(\gamma) = \int_{\bbR^d\times\bbR^d} |f(x)-f(y)| + \eps \|x-y\|_2 \, \dd \gamma(x,y) = \cE_0(\gamma) + \eps\underbrace{\int_{\bbR^d\times\bbR^d} \|x-y\|_2 \, \dd \gamma(x,y)}_{\leq M} \to \cE_0(\gamma) \]
as $\eps\to 0$.

\noindent\emph{Convergence of minimisers.}
Let $\gamma\in\Gamma(\mu,\nu)$, by the limsup part of the proof $\cE_\eps(\gamma)\to\cE_0(\gamma)$.
Combining with the liminf part of the proof implies 
\[ \cE_0(\gamma) = \lim_{\eps\to 0} \cE_\eps(\gamma) \geq \limsup_{\eps\to 0} \cE_\eps(\gamma_\eps^*) \geq \liminf_{\eps\to 0} \cE_\eps(\gamma_\eps^*) \geq \cE_0(\gamma^*) \geq \min_{\gamma\in\Gamma(\mu,\nu)} \cE_0(\gamma). \]
Taking the minimum over $\gamma\in\Gamma(\mu,\nu)$ gives
\[ \min_{\gamma\in\Gamma(\mu,\nu)} \cE_0(\gamma) \geq \limsup_{\eps\to 0} \cE_\eps(\gamma_\eps^*) \geq \liminf_{\eps\to 0} \cE_\eps(\gamma_\eps^*) \geq \cE_0(\gamma^*) \geq \min_{\gamma\in\Gamma(\mu,\nu)} \cE_0(\gamma). \]
Hence,
\[ \cE_0(\gamma^*) = \lim_{\eps\to 0} \cE_\eps(\gamma_\eps^*) = \min_{\gamma\in\Gamma(\mu,\nu)} \cE_0(\gamma) \]
and $\gamma^*$ is a minimiser of $\cE_0$.
In other words, $\gamma^*\in\argmin_{\eta\in\Gamma(\mu,\nu)} \langle \eta,D^{(f)}\rangle$.

Now let $\bar{\gamma}\in\Gamma(\mu,\nu)$ be any other minimiser of $\cE_0$.
We have,
\[ \eps \langle \bar{\gamma},C\rangle + \underbrace{\langle \bar{\gamma},D^{(f)}\rangle}_{=\cE_0(\bar{\gamma}) = \min_{\eta\in\Gamma(\mu,\nu)}\langle\eta,D^{(f)}\rangle} = \cE_\eps(\bar{\gamma}) \geq \cE_\eps(\gamma_\eps^*) = \eps\langle\gamma^*_\eps,C\rangle + \langle \gamma_\eps^*,D^{(f)}\rangle. \]
Rearranging this we have
\[ \eps \langle \bar{\gamma},C\rangle \geq \eps\langle\gamma^*_\eps,C\rangle + \underbrace{\langle \gamma_\eps^*,D^{(f)}\rangle - \min_{\eta\in\Gamma(\mu,\nu)}\langle\eta,D^{(f)}\rangle}_{\geq 0} \geq \eps\langle\gamma^*_\eps,C\rangle. \]
So, $\langle\bar{\gamma},C\rangle \geq \langle\gamma^*_\eps,C\rangle$.
Let $C_M(x,y) = \min\{\|x-y\|_2,M\}$, so $C_M$ is Lipschitz, bounded and $C_M\leq C$.
For fixed $M>0$,
\[ \langle \bar{\gamma},C\rangle \geq \langle \gamma_\eps^*, C\rangle \geq \langle\gamma_\eps^*,C_M\rangle \to \langle \gamma^*,C_M\rangle \]
as $\eps\to 0$.
Letting $M\to\infty$ and applying the monotone convergence theorem we have $\langle\bar{\gamma},C\rangle \geq \langle\gamma^*,C\rangle$.
Therefore,
\[ \langle \gamma^*,C\rangle = \min_{\gamma\in \displaystyle \underset{\eta\in\Gamma(\mu,\nu)}{\argmin}\langle \eta,D^{(f)}\rangle}\langle \gamma,C\rangle. \]

We can also write
\[ \eps\langle \gamma_\eps^*, C\rangle + \langle \gamma_\eps^*, D^{(f)}\rangle = \cE_\eps(\gamma_\eps^*) \leq \cE_\eps(\gamma^*) = \eps\langle \gamma^*,C\rangle + \langle \gamma^*,D^{(f)}\rangle. \]
Hence,
\[
\epsilon\langle\gamma_\epsilon^*,C\rangle
+
\left(
\langle\gamma_\epsilon^*,D^{(f)}\rangle
-
\langle\gamma^*,D^{(f)}\rangle
\right)
\leq
\epsilon\langle\gamma^*,C\rangle.
\]
Since $\gamma^*$ minimizes the sliced cost, the term in
parentheses is nonnegative. 
So $\langle \gamma_\eps^*,C\rangle \leq \langle \gamma^*,C\rangle$.
On the other hand, since $C$ is lower semi-continuous and
bounded from below, the Portmanteau theorem implies $\liminf_{\eps\to 0}\langle \gamma_\eps^*, C\rangle \geq \langle \gamma^*,C\rangle$.
Putting this together implies,
\[ \langle \gamma^*,C\rangle \geq \limsup_{\eps\to 0} \langle \gamma_\eps^*,C\rangle \geq \liminf_{\eps\to 0} \langle \gamma_\eps^*,C\rangle \geq \langle \gamma^*, C\rangle. \]
In particular, $\langle \gamma_\eps^*,C\rangle \to \langle \gamma^*, C\rangle$.
\end{proof}

\begin{remark}
In the liminf part of the proof we never used that $\gamma_\eps^*$ is a minimising sequence.
In particular, the proof shows that $\liminf_{\eps\to 0} \cE_\eps(\gamma_\eps) \geq \cE_0(\gamma)$ for all $\gamma_\eps\weakstarto \gamma$.
Combined with the limsup part of the proof this is (by definition) a $\Gamma$-convergence, with respect to the weak$^*$ topology. As a result $\Glim_{\eps\to 0} \cE_\eps= \cE_0$.
\end{remark}

\paragraph{A fixed perturbation is not exact in general}

For a fixed $\epsilon>0$, the perturbed linear program \eqref{eq: perturbed} is not automatically equivalent to the
exact lexicographic problem. 

We recall that for a fixed slicer $f$, the first-stage
problem is
\[
L(f)
=
\min_{\gamma\in\Gamma(\mu,\nu)}
\langle D^{(f)},\gamma\rangle.
\]
The tie-broken upper objective is then defined by the second-stage
problem
\[
U(f)
=
\min_{\gamma\in\Gamma(\mu,\nu)}
\left\{
\langle C,\gamma\rangle
:
\langle D^{(f)},\gamma\rangle=L(f)
\right\}.
\]
Equivalently, one first restricts to the face of slice-optimal
couplings $\Gamma_f$ and then minimizes the ambient cost over that face.

The perturbed formulation \eqref{eq: perturbed} instead computes
\begin{equation}\label{eq:LP_perturbed}
    \gamma_\epsilon
\in
\operatorname*{arg\,min}_{\gamma\in\Gamma(\mu,\nu)}
\left\{
\langle D^{(f)},\gamma\rangle
+
\epsilon\langle C,\gamma\rangle
\right\}.
\end{equation}
For a fixed finite-dimensional problem, the formulation \eqref{eq:LP_perturbed} is
guaranteed to recover the lexicographic solution only when $\epsilon$
is below an instance-dependent threshold. In particular, there is no
universal guarantee that a fixed $\epsilon>0$ is sufficiently small,
especially when $f_\theta$ changes during training and the gap between
the smallest and second-smallest sliced costs may approach zero.

As a concrete counterexample, let
\begin{equation}\label{eq: counterexample}
    \begin{cases}
       \mu=\tfrac12\left(\delta_{x_1}+\delta_{x_2}\right), & \text{ with } x_1=(0,0), \
x_2=(1,\delta), \ \\
\nu=\tfrac12\left(\delta_{y_1}+\delta_{y_2}\right), &
\text{ with }  
y_1=(1,0),
\
y_2=(0,\delta). 
    \end{cases}
\end{equation}
Consider the $1$-Lipschitz slicer $f(z)=z_2$
and take $\delta=5\times10^{-7}$.
There are two permutation couplings, whose sliced and ambient costs
are
\[
\begin{array}{c|c|c}
\text{Coupling}
&
\text{Sliced cost}
&
\text{Ambient cost}
\\ \hline
x_1\mapsto y_1,\quad x_2\mapsto y_2
&
0
&
1
\\
x_1\mapsto y_2,\quad x_2\mapsto y_1
&
\delta
&
\delta
\end{array}
\]
The exact lexicographic problem must select the first coupling because
its sliced cost is exactly minimal. Consequently,
$L(f)=0$,
$U(f)=1$.

Now take $\epsilon=10^{-6}$. The perturbed objective value in \eqref{eq:LP_perturbed} of the
first coupling is $0+\epsilon(1)=10^{-6}$,
whereas for the second coupling is $\delta+\epsilon\delta
=
5\times10^{-7}+5\times10^{-13}
\approx 5\times10^{-7}$.
Therefore, the perturbed linear program \eqref{eq:LP_perturbed} selects the second coupling,
even though it is not slice-optimal. Its ambient cost is
$5\times10^{-7}$, while the exact tie-broken value is $U(f)=1$.

More precisely, in this example the perturbed formulation selects the
lexicographic solution only if
\[
\epsilon
<
\frac{\delta}{1-\delta}
\approx 5\times10^{-7}.
\]
Thus, the fixed choice $\epsilon=10^{-6}$ is already too large.

For the radial slicer in the concentric-circle example in Section \ref{sec:experiments}, the
perturbation \eqref{eq:LP_perturbed} does recover the exact tie-broken solution because every
feasible coupling has exactly the same sliced cost. In that special
case, the perturbed problem reduces to minimizing the ambient cost, which is a special degeneracy.

\paragraph{Exact two-stage evaluation of $U(f)$ vs. the perturbed version} The discussion above shows
that a fixed $\epsilon$ cannot be relied on \eqref{eq:LP_perturbed} in general. 

To compute the lexicographic program for $U(f)$ numerically, we perform as follows: Stage one solves
$L(f) = \min_{\gamma \in \Gamma(\mu,\nu)} \langle D^{(f)}, \gamma \rangle$ by
network simplex. Stage two solves
\[
  U(f) \;=\; \min_{\gamma \in \Gamma(\mu,\nu)} \langle C, \gamma \rangle
  \qquad \text{subject to} \qquad
  \langle D^{(f)}, \gamma \rangle = L(f),
\]
as a linear program in which the primary optimum is imposed as an equality
constraint, so no perturbation parameter enters at any point.
Both cost matrices are rescaled to unit
maximum before the second solve. A solver's feasibility tolerance is absolute,
so on a configuration whose sliced costs are all of order $10^{-9}$ while its
ambient costs are of order $1$, the constraint $\langle D^{(f)},\gamma\rangle =
L(f)$ falls entirely inside the default tolerance and is satisfied by couplings
that are not slice-optimal at all; the solver then minimizes the ambient cost
over the whole transport polytope and returns a value far below $U(f)$. After
rescaling, the constraint is commensurate with the marginal constraints, and
the tolerance becomes relative. 
We set the primal and dual feasibility tolerances to \(10^{-10}\). Every evaluation reports the realized primary residual
$r_{\mathrm{primary}} = |\langle D^{(f)}, \gamma_U \rangle - L(f)|$, together with the marginal residuals, the smallest
entry of $\gamma_U$, and a check that $L(f) \le W_1 \le U(f)$. 
See Table~\ref{tab:upper-audit}.

As a numerical diagnostic,
with the
exact evaluator available, the estimate \eqref{eq:LP_perturbed} used in Figure \ref{fig:kr-experiments} can be checked: Table~\ref{tab:upper-audit}
recomputes $U$ at every recorded checkpoint of every experiment in the paper by
both routes. The two agree to at worst $3\times10^{-12}$ in relative terms and
select the \emph{same} coupling at every of the 64 checkpoint.
The perturbation is therefore adequate on these configurations (as an
empirical fact about them, not as a property of the construction).
\\

\begin{table}[ht!]
  \centering
  \small
  \begin{tabular}{llcccc}
    \hline
    Experiment & Objective &  $\max|U_\epsilon - U|$ & relative error &
    plans differ & $r_{\mathrm{primary}}$ \\
    \hline
    Gaussian blobs       & $J_{\mathrm{KR}}$ & $6.0\times10^{-15}$  & $5.3\times10^{-15}$  & 0\% & $2.7\times10^{-15}$ \\
                         & $L$               & $3.6\times10^{-15}$  & $3.1\times10^{-15}$  & 0\% & $1.6\times10^{-15}$ \\
    Two moons            & $J_{\mathrm{KR}}$  & $1.2\times10^{-14}$  & $1.7\times10^{-14}$  & 0\% & $5.6\times10^{-16}$ \\
                         & $L$                 & $1.1\times10^{-14}$  & $1.4\times10^{-14}$  & 0\% & $1.9\times10^{-15}$ \\
    Interleaved spirals  & $J_{\mathrm{KR}}$ & $5.9\times10^{-15}$  & $1.4\times10^{-14}$  & 0\% & $4.7\times10^{-15}$ \\
                         & $L$                & $1.4\times10^{-12}$  & $3.2\times10^{-12}$  & 0\% & $1.5\times10^{-14}$ \\
    Concentric circles   & $J_{\mathrm{KR}}$  & $2.2\times10^{-16}$  & $4.4\times10^{-16}$  & 0\% & $1.3\times10^{-16}$ \\
                         & $L$                & $0$                  & $0$                  & 0\% & $0$ \\
    \hline
  \end{tabular}
  \caption{The perturbed evaluation \eqref{eq:LP_perturbed} denoted by $U_\epsilon$ with
    $\epsilon = 5\times10^{-6}$ against the exact two-stage program $U$, at every
    recorded checkpoint (64) of every experiment in the paper; ``plans differ'' is
    the fraction of checkpoints at which the two routes return different
    couplings.}
  \label{tab:upper-audit}
\end{table}

Finally, we notice that the bound $\epsilon < \delta/(1-\delta)$
derived above is a statement about one pair of parameters, and it can be tested
over both. Figure~\ref{fig:perturbation-phase} evaluates the 
example \eqref{eq: counterexample} over $\delta \in \{10^{-9},\dots,10^{-2}\}$ and
$\epsilon \in \{10^{-10},\dots,10^{-2}\}$ and marks the cells where the
perturbed program \eqref{eq:LP_perturbed} returns a coupling that is not slice-optimal. It fails in
$29$ of the $72$ cells, and the measured pattern agrees with the analytic
boundary in $71$ of $72$; the single exception is $\delta = \epsilon = 10^{-8}$,
where the two sides of the comparison differ by only \(10^{-16}\) in absolute value, effectively producing a floating-point tie.

\begin{figure}[ht!]
  \centering
  \includegraphics[width=0.6\linewidth]{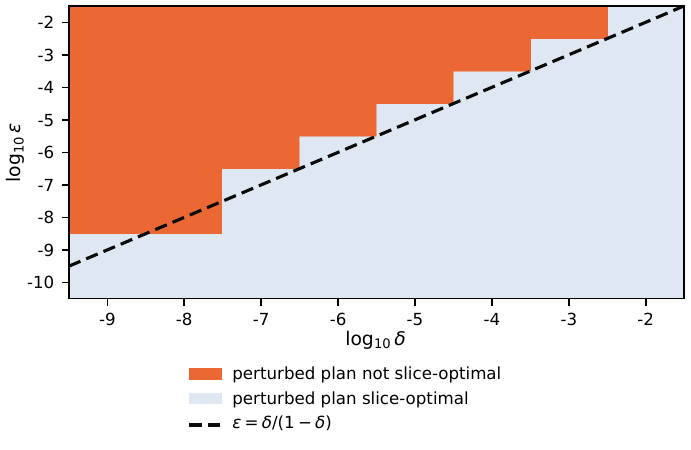}
  \vspace{-0.2in}
  \caption{Fixed perturbation optimization \eqref{eq: perturbed} (equivalently, \eqref{eq:LP_perturbed}), on the
    two-coupling example \eqref{eq: counterexample}. Shaded cells are those in which the perturbed program
    returns a coupling that is not slice-optimal; the dashed line is the
    analytic boundary $\epsilon = \delta/(1-\delta)$.}
  \label{fig:perturbation-phase}
\end{figure}

\section{Nonlinear Necessity}
\label{sec:app-nonlinear}

This subsection reports a quantitative study of when a nonlinear slicer is
actually needed. Every value of $U$ below is computed using the exact two-stage program, while every reported value of \(L\) is evaluated after certifying the slicer's Lipschitz constant using exact singular values computed by SVD. Every
evaluation was checked to satisfy $L \le W_1 \le U$; across the $740$ runs
reported here the largest primary residual is $2.0\times10^{-14}$.

\paragraph{A configuration with an analytic linear gap} Let $z_i \in
\mathbb{S}^{d-1}$ be shared directions and put $\mu = \frac1n\sum \delta_{R
z_i}$, $\nu = \frac1n \sum \delta_{r z_i}$. The radial matching is optimal, so
$W_1 = R - r$, and the radial slicer $f(x) = \|x\|_2$ closes the sandwich in
every dimension $d$. A linear slicer $f_v(x)=v^\top x$ cannot:  Both pushforwards are the same
scalars $t_i = \langle v, z_i\rangle$ scaled by $R$ and by $r$, so sorting
preserves the pairing and
\[
  \frac{L(f_v)}{W_1} \;=\; \frac1n \sum_i |\langle v, z_i\rangle|
  \;\xrightarrow[n\to\infty]{}\;
  c_d := \frac{\Gamma(d/2)}{\sqrt{\pi}\,\Gamma((d+1)/2)}
  \;\sim\; \sqrt{\tfrac{2}{\pi d}} .
\]
This is an analytic identity, not an approximation, so the optimized empirical direction
can be compared against a closed-form reference. 

Maximizing $\frac1n\sum_i
|\langle v, z_i\rangle|$ over the unit ball is a convex maximization; we use
the subgradient fixed point $v \mapsto \mathrm{normalize}(\sum_i
\mathrm{sign}(\langle v,z_i\rangle) z_i)$ from $32$ restarts.

Table~\ref{tab:sphere-linear} reports the optimized linear slicer over
$d \in \{2,\dots,128\}$, $n \in \{128,512,2048\}$, ten direction samples
each. The empirical value exceeds $c_d$ by a margin that grows with $d$ at
fixed $n$ and shrinks with $n$ at fixed $d$
and converges onto the analytic curve as $n$
grows. The oracle radial slicer closes the sandwich to a worst normalized gap
of $6.4\times10^{-13}$ over all $420$ runs, and $W_1 = R-r$ exactly in every
cell.

\begin{table}[ht!]
  \centering
  \small
  \begin{tabular}{rrccc}
    \hline
    $d$ & $c_d$ & $n=128$ & $n=512$ & $n=2048$ \\
    \hline
      2 & 0.6366 & 0.6790 ($\times1.07$) & 0.6502 ($\times1.02$) & 0.6481 ($\times1.02$) \\
      4 & 0.4244 & 0.4775 ($\times1.12$) & 0.4526 ($\times1.07$) & 0.4372 ($\times1.03$) \\
      8 & 0.2910 & 0.3626 ($\times1.25$) & 0.3256 ($\times1.12$) & 0.3086 ($\times1.06$) \\
     16 & 0.2026 & 0.2746 ($\times1.36$) & 0.2386 ($\times1.18$) & 0.2207 ($\times1.09$) \\
     32 & 0.1422 & 0.2176 ($\times1.53$) & 0.1800 ($\times1.27$) & 0.1612 ($\times1.13$) \\
     64 & 0.1001 & 0.1776 ($\times1.77$) & 0.1374 ($\times1.37$) & 0.1196 ($\times1.19$) \\
    128 & 0.0707 & 0.1457 ($\times2.06$) & 0.1086 ($\times1.54$) & 0.0900 ($\times1.27$) \\
    \hline
  \end{tabular}
  \caption{Optimized linear slicer on concentric spheres: median $L/W_1$ over
    ten direction samples, with the ratio to the analytic $c_d$ in parentheses.
    The ratio decreases monotonically in $n$ at every dimension.}
  \label{tab:sphere-linear}
\end{table}

\paragraph{Learned classes} Table~\ref{tab:sphere-classes} places the learned
classes between the two ends, at $n = 512$ over five seeds with a matched
budget of $2048$ iterations. Both GroupSort variants track the oracle far more
closely than any other class, and their advantage over the linear bound widens
with dimension: $1.6\times$, $3.1\times$, $6.1\times$ and $11.7\times$ at
$d = 2, 8, 32, 128$. The frozen random-feature class is \emph{worse} than
linear at $d=2$ and only overtakes it by $d=8$. The two training objectives are
indistinguishable here, unlike on the interleaved spirals: this geometry admits
no exact slicer that fails to be a potential, so there is nothing for direct
$L$-training to find.

Finally, the last column in Table \ref{tab:sphere-classes} shows that the upper bound is exactly $1$ for \emph{every} class at
\emph{every} dimension, including the linear slicer at $d = 128$ whose lower
bound has collapsed to $0.07$. This is a sharp form of one of the observations in Section \ref{sec:slice_sandwich}: $U(f) = W_1$ can hold for a slicer that has retained
almost none of the geometry, and only $L(f) = U(f)$ is informative.

\begin{table}[ht!]
  \centering
  \small
  \begin{tabular}{lccccc}
    \hline
    Class & $d=2$ & $d=8$ & $d=32$ & $d=128$ & $U/W_1$ \\
    \hline
    linear (analytic $c_d$)   & 0.6366 & 0.2910 & 0.1422 & 0.0707 & 1.0000 \\
    RFF (frozen)              & 0.4777 & 0.3839 & 0.3365 & 0.3017 & 1.0000 \\
    Fourier (learnable)       & 0.7701 & 0.6139 & 0.5691 & 0.5513 & 1.0000 \\
    GroupSort, $J_{\mathrm{KR}}$ & 0.9875 & 0.9113 & 0.8692 & 0.8263 & 1.0000 \\
    GroupSort, $L$            & 0.9868 & 0.9130 & 0.8679 & 0.8265 & 1.0000 \\
    \hline
  \end{tabular}
  \caption{Median $L/W_1$ over five seeds on concentric spheres at $n = 512$,
    with $U/W_1$ in the last column (identical for every class and dimension).
    Parameter counts and median runtimes: $2{,}049$ / $7$\,s for frozen
    features, $45{,}057$ / $15$\,s for learnable Fourier, $536{,}577$ /
    $40$\,s for GroupSort.}
  \label{tab:sphere-classes}
\end{table}

\section{Additional Experiment using RFF Slicers} \label{sec:app:exp-RFF}

We revisit the random Fourier feature (RFF) slicer class
$f_a(x)=a^\top z_{\Omega,b}(x)$ introduced in Section~\ref{sec:lip_nets}, in
which the frequencies $\omega_k$ and phases $b_k$ are sampled once and kept
fixed while only the coefficients $a\in\mathbb R^m$ are optimized, subject to
$\|a\|_2\leq\kappa_\Omega^{-1}$ with
$\kappa_\Omega=\sqrt{2/m}\,\|\Omega\|_{\mathrm{op}}$ as in
\eqref{eq:rff_lipschitz}, so that every admissible slicer is $1$-Lipschitz.
Section~\ref{sec:lip_nets} compared this class against a $1$-Lipschitz network
on a fixed target; here we use it instead to probe the factorization mechanism
of Theorem~\ref{thm: characterization}, by fitting the same class with two
different objectives.

For the same realization of $(\Omega,b)$, we fit the coefficients
using either the KR objective
\[
a_J\in
\mathrm{arg}\max_{\|a\|_2\leq\kappa_\Omega^{-1}}
J_{\mathrm{KR}}(f_a),
\qquad
J_{\mathrm{KR}}(f)
=
\mathbb E_\mu[f]-\mathbb E_\nu[f],
\]
or the sliced objective
\[
a_L\in
\mathrm{arg}\max_{\|a\|_2\leq\kappa_\Omega^{-1}}
L(f_a)
=
\mathrm{arg}\max_{\|a\|_2\leq\kappa_\Omega^{-1}}
W_1\bigl((f_a)_\#\mu,(f_a)_\#\nu\bigr).
\]
For a fixed feature dimension \(m=256\) and a single realization of
\((\Omega,b)\), Figure~\ref{fig:RFF} compares the normalized values
\(L(f_a)/W_1(\mu,\nu)\) and
\(|J_{\mathrm{KR}}(f_a)|/W_1(\mu,\nu)\) for the slicers obtained by optimizing the two respective objectives.
The purpose of this experiment is to determine whether a simple,
computationally inexpensive nonlinear slicer class can approximately
close the Wasserstein sandwich. The comparison between the two
objectives also tests the factorization mechanism of Theorem \ref{thm: characterization}:
maximizing $J_{\mathrm{KR}}(f)$ asks the slicer itself to approximate
an oriented KR potential, corresponding to the fixed outer function
$h(t)=t$ (i.e., we look for the direction $a_J$ that creates the largest difference of means of the scalar outputs), whereas maximizing $L(f_a)$ only asks $f_{a_L}$ to produce an
informative scalar representation, allowing an optimal
one-dimensional potential $h$ to complete the factorization
$h\circ f_{a_L}$ (i.e., the $L$-objective does not require $f_{a_L}$ itself to distinguish the two measures by a mean shift, instead it only requires 
$f_{a_L}$ to produce a scalar representation from which some $1$-Lipschitz $h$ can distinguish them).

\begin{figure}[ht!]
    \centering
    \includegraphics[width=0.9\linewidth]{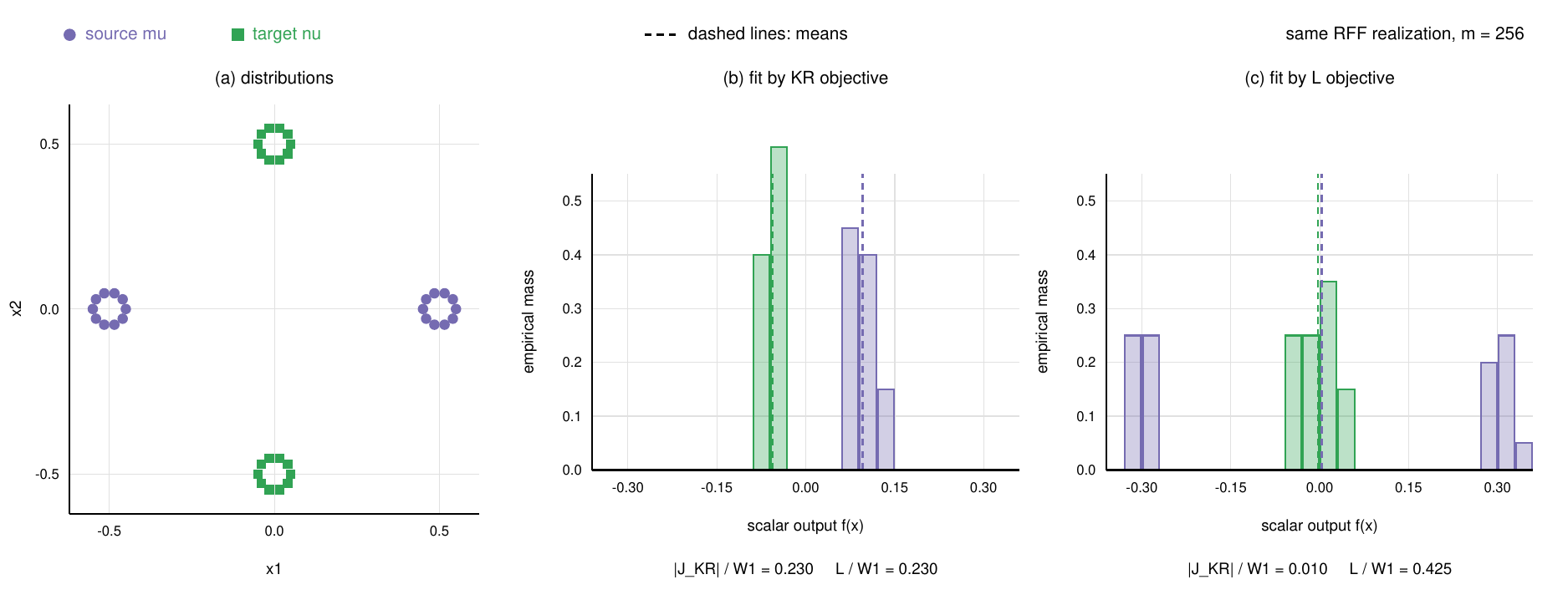}
    \vspace{-0.25in}
    \caption{
Direct comparison of RFF slicers trained with different objectives. Panel (a) shows the original two-dimensional
distributions. Panels (b) and (c) show histograms of the
one-dimensional scalar outputs $f(x)$ for the same RFF realization
with $m=256$; dashed lines indicate their means. Under KR training,
the two output distributions are separated by a mean shift, giving
$|J_{\mathrm{KR}}|/W_1=L/W_1=0.230$. 
Under $L$ training, the outputs
of $\mu$ form two outer groups while those of $\nu$ remain near zero.
The mean difference is therefore small,
$|J_{\mathrm{KR}}|/W_1=0.010$, although the complete distributions
are more strongly separated, with $L/W_1=0.425$. Hence, in this experiment, optimizing $L$
produces a more informative scalar representation for distinguishing $\mu$ and $\nu$ than optimizing
$J_{\mathrm{KR}}$ (in the sense that $f_{a_L}$ has encoded the distinction in the ``shape'' of the one-dimensional distributions rather than in their means).
}
    \label{fig:RFF}
\end{figure}

\section*{Acknowledgments}
The authors thank Vanderbilt University and Professor Akram Aldroubi for organizing the 9th International Conference on Computational Harmonic Analysis, where the initial ideas for this paper were developed. SK acknowledges funding from NSF CAREER Award \#2339898. RDM and SK acknowledge funding from NSF DMS Award \#2603773. MT acknowledges the support of Leverhulme Trust Research through the Project Award ``Robust Learning: Uncertainty Quantification, Sensitivity and Stability'' (grant agreement RPG-2024-051), the NHSBT award 177PATH25 ``Harnessing Computational Genomics to Optimise Blood Transfusion Safety and Efficacy'', the EPSRC Mathematical and Foundations of Artificial Intelligence Probabilistic AI Hub (grant agreement EP/Y007174/1), and the EPSRC-JST award ``Statistical Safeguarding: A Japan-UK Collaboration Towards the Responsible Data-driven Learning Paradigm''.

\bibliographystyle{plain}
\bibliography{sample}

\end{document}